\documentclass[10pt,a4paper]{article}
\usepackage[utf8]{inputenc}
\usepackage{amsmath}
\usepackage{amsfonts}
\usepackage{amssymb}
\usepackage{booktabs}
\usepackage{amsthm}
\usepackage{makeidx}
\usepackage{colortbl}%
\usepackage{xcolor}
\usepackage{enumerate}
\usepackage{graphicx}
\usepackage{multirow}
\usepackage[outdir=./]{epstopdf}
\usepackage{makecell} 
\usepackage[a4paper, total={6.5in, 8in}]{geometry}
\usepackage[linesnumbered]{algorithm2e}
\SetKwInput{init}{Initialization}
\usepackage{matlab-prettifier}
\usepackage{caption}
\usepackage{booktabs}  
\usepackage[column=0]{cellspace} 
\usepackage{float}
\theoremstyle{definition}
\newtheorem{definition}{Definition}

\theoremstyle{plain}
\newtheorem{theorem}{Theorem}

\theoremstyle{plain}
\newtheorem{lemma}{Lemma}

\theoremstyle{plain}
\newtheorem{remark}{Remark}

\theoremstyle{plain}

\theoremstyle{plain}
\newtheorem{assumptions}{Assumptions}

\usepackage[english]{babel}
\usepackage{csquotes}
\usepackage{authblk}
\usepackage{tikz,pgfplots}
\usepackage{lscape} 

\colorlet{tableheadcolor}{gray!25} 
\newcommand{\headcol}{\rowcolor{tableheadcolor}} %
\colorlet{tablerowcolor}{gray!10} 
\newcommand{\rowcol}{\rowcolor{tablerowcolor}} %
\newcommand{\topline}{\arrayrulecolor{black}\specialrule{0.1em}{\abovetopsep}{0pt}%
	\arrayrulecolor{tableheadcolor}\specialrule{\belowrulesep}{0pt}{0pt}%
	\arrayrulecolor{black}}
\newcommand{\midline}{\arrayrulecolor{tableheadcolor}\specialrule{\aboverulesep}{0pt}{0pt}%
	\arrayrulecolor{black}\specialrule{\lightrulewidth}{0pt}{0pt}%
	\arrayrulecolor{white}\specialrule{\belowrulesep}{0pt}{0pt}%
	\arrayrulecolor{black}}

\newcommand{\bottomlinec}{\arrayrulecolor{tablerowcolor}\specialrule{\aboverulesep}{0pt}{0pt}%
	\arrayrulecolor{black}\specialrule{\heavyrulewidth}{0pt}{\belowbottomsep}}%

\makeatletter
\renewcommand*\env@matrix[1][*\c@MaxMatrixCols c]{%
	\hskip -\arraycolsep
	\let\@ifnextchar\new@ifnextchar
	\array{#1}}
\makeatother

\usepackage[backend=bibtex,style=numeric-comp,giveninits=true,natbib]{biblatex}
\bibliography{SVMbib.bib}	

\makeatletter
\def\keywords{\xdef\@thefnmark{}}
\makeatother

\title{Proximal stabilized  Interior Point Methods for quadratic programming and \textit{low-frequency-updates} preconditioning techniques}
\author[1]{S. Cipolla \thanks{Email: \texttt{scipolla@ed.ac.uk}}}
\author[1]{J. Gondzio \thanks{Email: \texttt{j.gondzio@ed.ac.uk}}} 

\affil[1]{ \footnotesize{The University of Edinburgh, School of Mathematics}}

\begin{document}
	\maketitle
	
\begin{abstract}
In this work, in the context of Linear and Quadratic Programming, we interpret Primal Dual Regularized Interior Point Methods (PDR-IPMs) in the framework of the Proximal Point Method. The resulting Proximal Stabilized IPM (PS-IPM) is strongly supported by theoretical results concerning convergence and the rate of convergence, and can handle degenerate problems. Moreover, in the second part of this work, we analyse the interactions between the regularization parameters and the computational footprint of the linear algebra routines used to solve the Newton linear systems. In particular, when these systems are solved using an iterative Krylov method, we propose general purpose preconditioners which, exploiting the regularization and a new rearrangement of the Schur complement, remain attractive for a series of subsequent IPM iterations. Therefore they need to be recomputed only  in a fraction of the total IPM iterations. 
The resulting regularized second order methods, for which \textit{low-frequency-updates} of the preconditioners are allowed, pave the path for an alternative \textit{third way} in-between first and second order methods.   	
\end{abstract}

{ \footnotesize
	\noindent \keywords{\textbf{Keywords}: Interior point methods, Proximal point methods, Regularized primal-dual methods, Convex quadratic programming} \\
	\keywords{\textbf{MSC2010 Subject Classification:} 65K05, 90C51, 90C06}%
}

\section{Introduction}

In this work we consider the problem of solving the following primal-dual convex quadratic programs:
\begin{align} \label{eq:QP_problem}
\begin{aligned}
\min_{\mathbf{x} \in \mathbb{R}^d} \;& f(\mathbf{x}) :=\frac{1}{2}\mathbf{x}^TH\mathbf{x} +\mathbf{g}^T\mathbf{x}  \\ 
\hbox{s.t.} \; & A\mathbf{x}= \mathbf{b} \\
& \mathbf{x}_{\mathcal{C}}  \geq 0 , \; \mathbf{x}_{\mathcal{F}}  \hbox{ free } \\
\end{aligned}
\;\;\;\;\;\;\;\;\;\;
\begin{aligned}
	\max_{\mathbf{x} \in \mathbb{R}^d, \; \mathbf{y} \in \mathbb{R}^m, \; \mathbf{s}\in \mathbb{R}^{|\mathcal{C}|}} \;& \mathbf{b}^T\mathbf{y}-\frac{1}{2}\mathbf{x}^TH\mathbf{x} \\ 
	\hbox{s.t.} \; & 
	H\mathbf{x}+\mathbf{g}-A^T\mathbf{y}-\begin{bmatrix}
		0 \\
		\mathbf{s}
	\end{bmatrix}=0 \\
	& \mathbf{s}  \geq 0, \\
\end{aligned}
\end{align}
where $H \in \mathbb{R}^{d \times d},$ $H \succeq 0$, $A \in \mathbb{R}^{m \times d}$, $\mathcal{C}\subset \{ 1, \dots,d\}$ and $\mathcal{F}:=\{ 1, \dots,d\} \setminus \mathcal{C}$. $A$ is not required to have full rank but we assume that the condition $m \leq d$ holds. We assume, moreover, for simplicity of exposition and w.l.g., that $\mathcal{F}=\{1,\dots,\bar{d}\}$ and $\mathcal{C}=\{\bar{d}+1, \dots, d\}$ for some $\bar{d}<d$. 

For the past few decades, Interior Point Methods (IPMs) \cite{MR2881732, MR1422257} have gained wide appreciation due to their remarkable success in solving linear and convex quadratic programming problems \eqref{eq:QP_problem}. Computational cost of an IPM iteration is dominated by the solution of a
Karush-Kuhn-Tucker (KKT) linear system of the form

\begin{equation}\label{eq:general_KKT}
{\begin{bmatrix}
			H+D  &     -A^T 
				\\
			A & 0 
		\end{bmatrix}}\begin{bmatrix}
			\Delta \mathbf{x} \\
			\Delta \mathbf{y}  
		\end{bmatrix}=\begin{bmatrix}
			f_{\mathbf{x}} \\
			f_{\mathbf{y}}
		\end{bmatrix},
\end{equation}
where the diagonal matrix $D$ and the right hand side change at every iteration. The diagonal matrix $D$ represents, somehow, the core \textit{of the IPM methodology} and acts, essentially, as a continuous approximation of the indicator function for labelling active and non-active variables based on the magnitude of its diagonal elements: in the limit, these elements approach $0$ or $+\infty$. 

A closer look at the KKT system in \eqref{eq:general_KKT}, reveals how the astonishing polynomial worst-case complexity of IPMs \cite{MR2881732, MR1422257} is counterbalanced by an intrinsic difficulty for the optimal tuning of the linear algebra solvers required for their implementation. We briefly summarize two such issues: 

\begin{enumerate}[({I}1)]
	\item \label{item:I1} near rank deficiency of $A$, or near singularity of $H + D$, can give rise to inefficient or unstable solutions of the KKT linear systems.  This may occur when both direct or iterative methods are used for their solution. Moreover, it is important to mention at  this point, the related issue concerning the fact that in case of a rank deficiency of $A$, the theory of convergence for IPM is not clear.
	\item \label{item:I2} for large scale problems, the unwelcome feature displayed by the diagonal elements of $D$  represents a paradigmatic example of how Krylov methods may be easily made
	ineffective due to the fact that the conditioning of the involved linear systems deteriorates as the IPM iterations proceed. As a result, the robustness and efficiency of IPMs depend heavily on the use of preconditioners which should be recomputed/updated at every iteration due to the presence of the rapidly changing matrix $D$. 
\end{enumerate}  

\subsection{Motivations and background}
In the last $20$ years there has been an intense research activity regarding items (I\ref{item:I1}) and (I\ref{item:I2}) mentioned in the previous section. In particular, in order to alleviate some of the numerical difficulties related to (I\ref{item:I1}), it has been proposed in  \cite{MR1777460} to systematically modify the linear system \eqref{eq:general_KKT} using, in essence, a diagonal regularization. Despite the fact that this strategy has proven to be effective in practice, to the best of our knowledge, in literature few works are devoted to the complete theoretical understanding of these regularization techniques. We mention \cite{MR2899152}, where the global convergence of an \textit{exact primal–dual regularized} IPM is shown under the somehow strong hypothesis that the computed Newton directions are uniformly bounded (see \cite[Th. 5.4]{MR2899152}) and \cite{MR4215213}, where regularization is interpreted in connection with the Augmented Lagrangian Method \cite{MR271809} and the Proximal Method of Multipliers \cite{MR418919} and where the convergence to the true solution is recovered when the regularization is driven to zero at a suitable speed.  

Concerning (I\ref{item:I2}), the literature is quite extensive and it is not possible to give a short comprehensive outlook.  We refer the interested reader to \cite[Sec. 5]{MR2881732} and \cite{MR2594606} for a comprehensive survey.  We prefer to stress, instead, the fact that the presence of the iteration dependent matrix $D$ and its diverging elements represents, somehow, the true  challenge in the efficient implementation of IPMs for large scale problems. As a matter of fact, the computational costs related to the \textit{necessary} re-computations of factorizations and/or preconditioners for the Newton linear systems represent the main bottleneck of the existing implementations.  

This work aims at using regularization as the tool to solve/alleviate simultaneously  the issues outlined in items (I\ref{item:I1}) and (I\ref{item:I2}). Indeed, broadly speaking, this work can be viewed as a study of the interactions between the regularization parameters used in the Primal Dual Regularized Interior Point Methods (PDR-IPMs) and the computational footprint of the linear algebra routines used to solve the related Newton linear systems.

\subsection{Contribution and organization}

The investigation which aims at addressing both issues (I\ref{item:I1}) and (I\ref{item:I2}) is naturally organized into two main threads. Indeed, in the first part of this work we aim at clarifying how alleviating the (near) rank deficiency of matrix $A$ using regularization affects the convergence of the underlying IPM scheme.
To this end, we build a bridge between Primal Dual Regularized IPMs (PDR-IPMs)  and the Proximal Point Method (PPM) \cite{MR410483} giving a precise description of the synergies occurring between them. 
In particular, our analysis contributes to the understanding of the hidden machinery which controls the convergence of the PDR-IPM and clarifies, finally, the influence of regularization for PDR-IPMs:  our proposal, the Proximal Stabilized IPM (PS-IPM), is strongly supported by theoretical results concerning convergence/rate of convergence and can handle degenerate problems as those described in (I\ref{item:I1}). 

In the second part of this work, building the momentum from the developed convergence theory, we address (I\ref{item:I2}) using a PS-IPM perspective. Here we prove that regularization can be used, in fact, as a tool to
pursue the challenging aspiration of reducing systematically the number of necessary preconditioner re-computations needed for the iterative solution of IPM Newton linear systems. Indeed, using an equivalent formulation of the LP/QP problem and a new rearrangement of the Schur complements for the related Newton systems, we are able to prove that when such systems are solved using an iterative Krylov method, the number of necessary preconditioner re-computations  equals just  a fraction of the total IPM iterations. As a straightforward consequence of the above findings, we are able to devise a class of IPM-type methods characterized by the fact that the re-computation of any given preconditioner can be performed in \textit{low-frequency}
regime, hence the linear algebra footprint of the method is significantly reduced.

The precise outline of the contribution and the organization of the work can be summarized as follows. 

In Sections \ref{sec:convex_formulation} and \ref{sec:PD_REG_IPM}, using the PPM \cite{MR410483} in its inexact variant \cite{MR732428}, we  prove the convergence of the PDR-IPM-type schemes for the solution of problem \eqref{eq:QP_problem} without any further assumptions on the uniform boundedness of the Newton directions or assuming that the regularization is driven to zero. Indeed, if on one hand our PS-IPM sheds further light
on the experimental evidence that regularization is of extreme importance  for robust and efficient IPMs implementations, on the other, it is supported by a precise result, see Theorem \ref{th:convergence_theorem}, tying the magnitude of the regularization parameters to its rate of convergence. The experimental evidence of the goodness of the proposed framework and of the resulting implementation is presented in Section~\ref{sec:num_res_p1} where we show that, when direct methods are used for the solution of the Newton system, fixing the regularization parameters to small values is preferable to a strategy in which a decreasing sequence of regularization parameters is employed.   

In Section \ref{sec:slac_prec}, we are able to depict a precise quantitative picture on the relations intervening between the regularization parameters and the necessity of recomputing any given preconditioner. Indeed, heavily relying on the form of the primal-dual regularized Newton systems and using a novel rearranging of their Schur complement which is based by a \textit{separation of variables trick}, we propose and analyse a new preconditioning technique for which the \textit{frequency of re-computation} depends inversely  on the magnitude of the regularization parameters. As a result, in the proposed PS-IPM scheme, the overall computational footprint of the linear algebra solvers can be decreased at the cost of slowing  down its rate of convergence.  
  
Finally, in Section \ref{sec:num_res_p2}, we carry out an experimental analysis of PS-IPMs when the corresponding Newton systems are solved using a Krylov iterative method precoditioned as proposed in Section \ref{sec:slac_prec}. The presented results show that our proposal can be tuned to obtain a number of preconditioner updates  roughly equal to one third of the total IPM iterations while maintaining an IPM-type rate of convergence, leading, hence, to significant improvements of the computational performance in large scale settings.

\section{Convex Formulation and the Proximal Point Algorithm} \label{sec:convex_formulation}

For problem \eqref{eq:QP_problem} we consider the following Lagrangian function:

\begin{equation} \label{eq:Lag_convex_for}
\mathcal{L}(\mathbf{x}, \mathbf{y})=\frac{1}{2}\mathbf{x}^TH\mathbf{x} +\mathbf{g}^T\mathbf{x}-\mathbf{y}^T(A
\mathbf{x} -\mathbf{b}) + I_D(\mathbf{x}, \mathbf{y}),
\end{equation}
where $I_D(\mathbf{x}, \mathbf{y})$ is the indicator function of the convex closed set
$$D:=\mathbb{R}^{\bar{d}} \times \mathbb{R}_{\geq 0}^{d-\bar{d}} \times \mathbb{R}^{m}.$$
\begin{lemma} \label{lem:boundary}
We have that	$\mathcal{B}(D)=\mathbb{R}^{\bar{d}} \times \mathcal{B}(\mathbb{R}_{\geq 0}^{d-\bar{d}}) \times \mathbb{R}^{m}$
where $\mathcal{B}(\cdot)$ is the boundary operator.
\end{lemma}
\begin{proof}
	Let us denote by $Cl(\cdot)$ the closure operator and observe that $$\mathcal{B}(A \times B \times C)=(\mathcal{B}(A)\times {C}l({B}) \times {C}l({C}) ) \cup ({C}l({A}) \times \mathcal{B}({B}) \times {C}l({C}) ) \cup ({C}l(A)\times {C}l({B}) \times \mathcal{B}({C}) ). $$
	Thesis follows observing that $\mathcal{B}(\mathbb{R}^{\bar{d}})=\mathcal{B}(\mathbb{R}^{{m}})=\emptyset$ and that $\mathbb{R}^{\bar{d}}$ and $\mathbb{R}^{\bar{m}}$ are closed.
\end{proof}

Let us now define  the saddle sub-differential operator related to \eqref{eq:Lag_convex_for} as

\begin{equation} \label{eq:saddle_sub_diff}
	T_\mathcal{L}(\mathbf{x},\mathbf{y}):=\begin{bmatrix}
	\partial_{\mathbf{x}}\mathcal{L}(\mathbf{x},\mathbf{y}) \\
	\partial_{\mathbf{y}}(-\mathcal{L}(\mathbf{x},\mathbf{y}))
	\end{bmatrix}= \begin{bmatrix}
	H\mathbf{x}+\mathbf{g}-A^T\mathbf{y} + \partial_\mathbf{x} I_D(\mathbf{x},\mathbf{y}) \\
	A\mathbf{x}-\mathbf{b}+\partial_\mathbf{y} I_D(\mathbf{x},\mathbf{y})
	\end{bmatrix},
\end{equation} 
where $\partial_{\mathbf{x}}$, $\partial_{\mathbf{y}}$ represent the partial sub-differential operators. The proper saddle function $\mathcal{L}(\mathbf{x},\mathbf{y})$ satisfies the hypothesis of \cite[Cor. 2]{rockafellar1970monotone} and hence the associated saddle operator, namely $T_{\mathcal{L}}$, is maximal monotone. In particular, the solutions $[\mathbf{x}^*,\mathbf{y}^*]$ of the problem $0 \in T_\mathcal{L}(\mathbf{x},\mathbf{y})$ are just the saddle points of $\mathcal{L}$, if any.

\begin{assumptions} We assume that the set $T_{\mathcal{L}}^{-1}(0) \neq \varnothing $ and $[\mathbf{x}^*,\mathbf{y}^*]^T \in \mathbb{R}^{d+m}$ represents a generic point in $T_{\mathcal{L}}^{-1}(0)$. Moreover, since $T_{\mathcal{L}}$ is maximal monotone, $T_{\mathcal{L}}^{-1}(0)$ is closed and convex. 
\end{assumptions}

Given a convex set $C \subset \mathbb{R}^{u}$, let us denote by $N_C(\mathbf{u})$ the normal cone to $C$ in $\mathbf{u} \in \mathbb{R}^u$ (see \cite[Sec. 2.1]{MR2515104}). In our case, we have that $$N_D(\mathbf{x},\mathbf{y})=\partial_\mathbf{x} I_D(\mathbf{x},\mathbf{y})\times \partial_\mathbf{y} I_D(\mathbf{x},\mathbf{y}) = N_{\mathbb{R}^{\bar{d}} \times \mathbb{R}_{\geq 0}^{d-\bar{d}}}(\mathbf{x})\times N_{\mathbb{R}^{m}}(\mathbf{y}). $$

The problem of finding $[\mathbf{x}^*,\mathbf{y}^*]$ s.t. $0 \in T_\mathcal{L}(\mathbf{x}^*,\mathbf{y}^*)$ can be alternatively written as the one of finding a solution for the problem
\begin{equation} \label{eq:var_formulation}
	- \begin{bmatrix}
	H & -A^T \\
	A & 0
	\end{bmatrix}\begin{bmatrix}
	\mathbf{x}\\
	\mathbf{y}
	\end{bmatrix} +\begin{bmatrix}
	-\mathbf{g} \\
	\mathbf{b}
	\end{bmatrix} \in N_D(\mathbf{x},\mathbf{y}),
\end{equation}
which represents the variational inequality formulation of problem \eqref{eq:saddle_sub_diff} (see \cite[Sec. 2.1]{MR2515104}). Moreover, using Lemma \ref{lem:boundary}, we have that

\begin{equation*}
	[\mathbf{v},\mathbf{w}] \in N_{D}(\mathbf{x},\mathbf{y}) \Leftrightarrow \begin{cases}
	 \mathbf{v}_i=0 \hbox{ for } i = 1,\dots, \bar{d} \\
	 \mathbf{v}_i=0  \hbox{ if } \mathbf{x}_i>0 \hbox{ and } i=\bar{d}+1, \dots, {d}\\
	 \mathbf{v}_i \leq 0  \hbox{ if } \mathbf{x}_i =0 \hbox{ and } i=\bar{d}+1, \dots, {d}\\
	 \mathbf{w}_i=0 \hbox{ for } i = 1,\dots, m \\
	\end{cases}.
\end{equation*}

At this stage, it is important to observe that given $[\mathbf{x}^*,\mathbf{y}^*]$ a solution of \eqref{eq:var_formulation}, we can recover a solution $[\mathbf{x}^*,\mathbf{y}^*,\mathbf{s}^* ]$  of \eqref{eq:QP_problem} defining $\mathbf{s}^*~:=-\mathbf{v}^*(\bar{d}+1:d)$ where $[\mathbf{v}^*,\mathbf{w}^*] \in N_{D}(\mathbf{x}^*,\mathbf{y}^*)$.

\subsection{Proximal Point Method}
In this section we follow essentially the developments from \cite{MR4047487,de2020primal} specializing our discussion for the operator $T_{\mathcal{L}}$. The Proximal Point Method (PPM) \cite{MR410483} finds zeros of maximal monotone operators by recursively applying their proximal operator. In particular, starting from an initial guess $[\mathbf{x}_0, \mathbf{y}_0]$, a sequence $[\mathbf{x}_k, \mathbf{y}_k]$ of primal-dual pairs is generated as follows:

\begin{equation} \label{eq:PPM_TL}
(\mathbf{x}_{k+1}, \mathbf{y}_{k+1})=\mathcal{P}(\mathbf{x}_{k}, \mathbf{y}_{k}),  \hbox{ where } \mathcal{P} = (I + \Sigma^{-1} T_{\mathcal{L}})^{-1} \hbox{ and } \Sigma:= blockdiag(\rho I_d , \delta I_m).
\end{equation}
Since $\Sigma^{-1} T_{\mathcal{L}}$ is yet maximal monotone, the operator $\mathcal{P}$ is single valued, non expansive and the generated sequence converges to a solution $[\mathbf{x}^*,\mathbf{y}^*] \in T_{\mathcal{L}}^{-1}(0)$ \cite{MR410483}.  

Evaluating the proximal operator $\mathcal{P}$ is equivalent of finding a solution to the problem

\begin{equation*}
	0 \in T_{\mathcal{L}}(\mathbf{x},\mathbf{y})+ \Sigma((\mathbf{x},\mathbf{y})-(\mathbf{x}_k,\mathbf{y}_k)),
\end{equation*}
which is guaranteed to have a unique solution. In particular, evaluating the proximal operator is equivalent to finding a solution of 

\begin{equation} \label{eq:regularized_saddle_formulation}
0 \in 	\begin{bmatrix}
	H\mathbf{x}+\mathbf{g}-A^T\mathbf{y} + \partial_\mathbf{x} I_D(\mathbf{x},\mathbf{y}) +\rho (\mathbf{x}-\mathbf{x}_k) \\
	A\mathbf{x}-\mathbf{b}+\partial_\mathbf{y} I_D(\mathbf{x},\mathbf{y}) + \delta (\mathbf{y}-\mathbf{y}_k)
	\end{bmatrix}
\end{equation}
which, in turn, corresponds to solving the primal dual regularized problem in \eqref{eq:QP_R_problem}:

\begin{align} \tag{RP} \label{eq:QP_R_problem}
\begin{aligned}
\min_{\mathbf{x} \in \mathbb{R}^d} \;& \frac{1}{2}\mathbf{x}^TH\mathbf{x} +\mathbf{g}^T\mathbf{x}+\frac{\rho}{2}\|\mathbf{x}-\mathbf{x}_k\|^2 +\frac{\delta}{2}\|\mathbf{y}\|^2  \\ 
\hbox{s.t.} \; & A\mathbf{x}+\delta(\mathbf{y}-\mathbf{y}_k)= \mathbf{b} \\
& \mathbf{x}_{\mathcal{C}}  \geq 0 , \; \mathbf{x}_{\mathcal{F}}  \hbox{ free }. \\
\end{aligned}
\end{align}
Moreover, also in this case, \eqref{eq:QP_R_problem} can be written in the following variational form:

\begin{equation} \label{eq:Reg_var_formulation}
- \begin{bmatrix}
H+\rho I & -A^T \\
A & \delta I
\end{bmatrix}\begin{bmatrix}
\mathbf{x}\\
\mathbf{y}
\end{bmatrix} +\begin{bmatrix}
-\mathbf{g} +\rho \mathbf{x}_k\\
\mathbf{b}+\delta \mathbf{y}_k
\end{bmatrix} \in N_D(\mathbf{x},\mathbf{y}).
\end{equation}

\subsection{Inexact PPM}
The inexact PPM has been originally analysed in \cite{MR410483} but we follow here the developments of \cite{MR732428}. We consider an approximate version of the PPM scheme in \eqref{eq:PPM_TL} where $(\mathbf{x}_{k+1}, \mathbf{y}_{k+1})$ satisfies
the criterion $(B)$ in \cite{MR732428}, i.e.,

\begin{equation} \label{eq:Inexat_crit}
	\|(\mathbf{x}_{k+1}, \mathbf{y}_{k+1})-\mathcal{P}(\mathbf{x}_{k}, \mathbf{y}_{k})\| \leq \varepsilon_k \min\{1, \|(\mathbf{x}_{k+1}, \mathbf{y}_{k+1})-(\mathbf{x}_{k}, \mathbf{y}_{k}) \|\}, \hbox{ where } \sum_{k=0}^{+\infty} \varepsilon_k < \infty.
\end{equation}

\begin{definition}
	In general, given $\mathbf{z} \in \mathbb{R}^{m+n}$ and a closed set $C$  we define
	\begin{equation*}
	dist(\mathbf{z},C):=\inf\{\|\mathbf{z}-{\mathbf{c}}\| \hbox{ for } {\mathbf{c}}\in {C} \}.
	\end{equation*}
\end{definition}
Theorem \ref{th:convergence_theorem} summarizes the results we are going to use in this work (the statements are specialized for our case):

\begin{theorem} \label{th:convergence_theorem}
	\begin{enumerate}
		\item Let us define the operator $S_k(\mathbf{x},\mathbf{y}):=T_{\mathcal{L}}(\mathbf{x},\mathbf{y})+\Sigma ((\mathbf{x},\mathbf{y})-(\mathbf{x}_k,\mathbf{y}_k))$ (see equation \eqref{eq:regularized_saddle_formulation}), then the condition in \eqref{eq:Inexat_crit} is implied by the condition 
		\begin{equation} \label{eq:Inexat_crit2}
			dist(0, S_k(\mathbf{x}_{k+1},\mathbf{y}_{k+1})) \leq {\min({\rho,\delta})}{\varepsilon_k}\min\{1, \|(\mathbf{x}_{k+1}, \mathbf{y}_{k+1})-(\mathbf{x}_{k}, \mathbf{y}_{k}) \|\},
		\end{equation}
		see \cite[Prop. 3]{MR410483}.
		\item The sequence $\{(\mathbf{x}_{k},\mathbf{y}_{k})\}_{k \in \mathbb{N}}$ generated by the recursion in \eqref{eq:PPM_TL} and using as inexactness criterion
		
		\begin{equation*} 
		\|(\mathbf{x}_{k+1}, \mathbf{y}_{k+1})-\mathcal{P}(\mathbf{x}_{k}, \mathbf{y}_{k})\| \leq \varepsilon_k, \hbox{ where } \sum_{k=0}^{+\infty} \varepsilon_k < \infty,
		\end{equation*}
		
		 is bounded if and only if there exists at least one solution of the problem $0 \in T_{\mathcal{L}}(\mathbf{x},\mathbf{y})$. Moreover it converges in the weak topology to a point $(\mathbf{x}^{*},\mathbf{y}^{*}) \in T_{\mathcal{L}}^{-1}(0)$ and
		
		\begin{equation*}
			0 = \lim_{k \to \infty} \|(I-\mathcal{P})(\mathbf{x}_{k},\mathbf{y}_{k})\|=\lim_{k \to \infty} \|(\mathbf{x}_{k+1},\mathbf{y}_{k+1})-(\mathbf{x}_{k},\mathbf{y}_{k})\|,
		\end{equation*}
		see \cite[Th. 1]{MR410483}.
		
		\item Suppose that
		
		\begin{equation} \label{eq:ZeroLipshtz}
			\exists \; a>0, \; \exists \; \delta>0 : \; \forall \mathbf{w} \in B(0,\delta), \; \forall \mathbf{z} \in T^{-1}_{\mathcal{L}}(\mathbf{w}) \hbox{ we have } dist(\mathbf{z} - T_{\mathcal{L}}^{-1}(0)) \leq a \|\mathbf{w}\|.
		\end{equation}
		Then, the sequence $\{(\mathbf{x}_{k},\mathbf{y}_{k})\}_{k \in \mathbb{N}}$ generated by the recursion in \eqref{eq:PPM_TL} using as inexactness criterion \eqref{eq:Inexat_crit}, is such that $dist((\mathbf{x}_{k},\mathbf{y}_{k}),T_{\mathcal{L}}^{-1}(0))\to 0$ linearly. Moreover, the  rate of convergence  is bounded by $a/(a^2+ (1/\max\{\rho,\delta\})^2)^{1/2}$, i.e.,
		
		\begin{equation} \label{eq:rate_of_convergence}
			\lim\sup_{k \to \infty} \frac{dist((\mathbf{x}_{k+1},\mathbf{y}_{k+1}),T_{\mathcal{L}}^{-1}(0))}{dist((\mathbf{x}_{k},\mathbf{y}_{k}),T_{\mathcal{L}}^{-1}(0))} \leq  \frac{a}{(a^2+(1/\max\{\rho,\delta\})^2)^{1/2}}<1,
		\end{equation}
		
		 see \cite[Th. 2.1]{MR732428}.
		
		\item The operators $T_{\mathcal{L}}$ and $T_{\mathcal{L}}^{-1}$ are polyhedral variational inequalities and hence they satisfy condition \eqref{eq:ZeroLipshtz}, see \cite[Sec. 3.4]{MR2515104}.
%
	\end{enumerate}
\end{theorem}

\subsection{Inexact PPM: practical stopping criteria}

As stated in Item 1. of Theorem \ref{th:convergence_theorem}, in order to guarantee linear convergence, we need to impose algorithmically the condition in \eqref{eq:Inexat_crit2}.   In particular, using \eqref{eq:Reg_var_formulation} and the fact that, in general, it holds $$\mathbf{v} \in N_D(\mathbf{x}) \Leftrightarrow \Pi_D(\mathbf{x}+\mathbf{v})= \mathbf{x},$$
see \cite[Sec. 2.1]{MR2515104}, we can define the following \textit{natural residual} (used also in \cite{de2020primal,MR4047487}):

\begin{equation} \label{eq:PPM_res_reg}
	r_k(\mathbf{x},\mathbf{y}):=\begin{bmatrix}
		\mathbf{x} \\
		\mathbf{y}
	\end{bmatrix} - \Pi_{D}(\begin{bmatrix}
	\mathbf{x} \\
	\mathbf{y}
\end{bmatrix}- \begin{bmatrix}
		H\mathbf{x}+\mathbf{g}-A^T\mathbf{y} +\rho (\mathbf{x}-\mathbf{x}_k)  \\
		A\mathbf{x}-\mathbf{b}+ \delta (\mathbf{y}-\mathbf{y}_k) 
	\end{bmatrix}).
\end{equation} 

Using analogous reasoning as in the proof \cite[Prop. 2, Item 3.]{MR4047487} we state the existence of a constant $\tau_1>0$ s.t.

\begin{equation*}
	dist(0, S_k(\mathbf{x},\mathbf{y})) \leq \tau_1 \| r_k(\mathbf{x},\mathbf{y})\|.
\end{equation*} 
Analogously, defining 

\begin{equation} \label{eq:PPM_res}
	r(\mathbf{x},\mathbf{y}):=\begin{bmatrix}
		\mathbf{x} \\
		\mathbf{y}
	\end{bmatrix} - \Pi_{D}(\begin{bmatrix}
		\mathbf{x} \\
		\mathbf{y}
	\end{bmatrix}- \begin{bmatrix}
		H\mathbf{x}+\mathbf{g}-A^T\mathbf{y} \\
		A\mathbf{x}-\mathbf{b} 
	\end{bmatrix}),
\end{equation} 
we have $dist(0,T_{\mathcal{L}}^{-1}(\mathbf{x},\mathbf{y}))=O(\|r(\mathbf{x},\mathbf{y})\|)$. 

In Algorithm \ref{alg:Inexact_PPM} we present the particular form of the inexact PPM considered in this work.

\begin{algorithm}[hbt!]
	\caption{Inexact PPM for QP}\label{alg:Inexact_PPM}
	\KwIn{ $tol > 0$, $\sigma_r \in (0,1)$.  }
	\init{Iteration counter $k = 0$; initial points $\mathbf{x}_0,\;\mathbf{y}_0$}
	\While{$\|r(\mathbf{x},\mathbf{y})\|>tol$}{
		Find $(\mathbf{x}_{k+1},\mathbf{y}_{k+1})$ s.t. $\|r_k(\mathbf{x}_{k+1},\mathbf{y}_{k+1}))\| < \frac{\min({\rho,\delta})}{\tau_1}\sigma_r^k \min\{1, \|(\mathbf{x}_{k+1}, \mathbf{y}_{k+1})-(\mathbf{x}_{k}, \mathbf{y}_{k}) \|\} $ \\
		Update the iteration counter: $k := k + 1$.
	}
\end{algorithm}

\section{Primal-dual IPM for Proximal Point evaluations} \label{sec:PD_REG_IPM}

For problem \eqref{eq:QP_R_problem} let us introduce the Lagrangian

\begin{equation} \tag{RL} \label{eq:Lagrangian_R_Classic1}
\mathcal{L}_k(\mathbf{x}, \mathbf{y}, \mathbf{s})=\frac{1}{2}[\mathbf{x}^T, \mathbf{y}^T] \begin{bmatrix}
H+\rho I & 0 \\
0 & \delta I
\end{bmatrix}\begin{bmatrix}
\mathbf{x} \\
\mathbf{y}
\end{bmatrix} +[\mathbf{g}^T- \rho \mathbf{x}_k^T, 0 ]\begin{bmatrix}
\mathbf{x} \\
\mathbf{y}
\end{bmatrix}-\mathbf{y}^T(A
\mathbf{x} + \delta (\mathbf{y}-\mathbf{y}_k) -\mathbf{b}) - \mathbf{s}^T\mathbf{x}_{\mathcal{C}},
\end{equation}
where $\mathbf{s} \in \mathbb{R}^{|\mathcal{C}|}$ and $\mathbf{s}\geq 0$. Using \eqref{eq:Lagrangian_R_Classic1}, we write the KKT conditions
	\begin{align*} 
	\begin{bmatrix}
	H+\rho I & 0 \\
	0 & \delta I
	\end{bmatrix}\begin{bmatrix}
	\mathbf{x} \\
	\mathbf{y}
	\end{bmatrix}  +\begin{bmatrix}
	\mathbf{g} - \rho\mathbf{x}_k \\
	0
	\end{bmatrix}-\begin{bmatrix}
	A^T \mathbf{y} \\
	\delta \mathbf{y} +(A\mathbf{x}+\delta(\mathbf{y}-\mathbf{y}_k)-\mathbf{b})
	\end{bmatrix}  -\begin{bmatrix}
	0 \\
	\mathbf{s} \\
	0
	\end{bmatrix}=0; \\ 
	SX_{\mathcal{C}}\mathbf{e}=0; \\
	\mathbf{x}_{\mathcal{C}} \geq 0.
	\end{align*}

We can then  write  the dual form of problem \eqref{eq:QP_R_problem} as

\begin{align} \tag{RD} \label{eq:Dual_R_QP_problem}
\begin{aligned}
\max_{\mathbf{x} \in \mathbb{R}^d, \; \mathbf{y} \in \mathbb{R}^m, \; \mathbf{s}\in \mathbb{R}^{|\mathcal{C}|}} \;& \mathbf{y}^T\mathbf{b}  
- \frac{1}{2}\mathbf{x}^TH\mathbf{x} -\frac{\rho}{2}\|\mathbf{x}\|-\frac{\delta}{2}\|\mathbf{y}-\mathbf{y}_k\| \\ 
\hbox{s.t.} \; & 
(H+\rho I)\mathbf{x}+(\mathbf{g} - \rho\mathbf{x}_k)-A^T\mathbf{y} -\begin{bmatrix}
0 \\
\mathbf{s} 
\end{bmatrix}=0 \\
& \mathbf{s}  \geq 0, \\
\end{aligned}
\end{align}
where we used the fact that $(A\mathbf{x}+\delta\mathbf{y})=\mathbf{b}+\delta \mathbf{y}_k$. 

\begin{lemma} \cite[Lem 3.1]{MR2881732} If $(\mathbf{x},\mathbf{y},\begin{bmatrix}
	0 \\
	\mathbf{s} 
	\end{bmatrix})$ is primal dual feasible, then the duality gap is equal to the complementarity gap, i.e.,
	\begin{equation*}
		\frac{1}{2}\mathbf{x}^TH\mathbf{x} +\mathbf{g}^T\mathbf{x}+\frac{\rho}{2}\|\mathbf{x}-\mathbf{x}_k\|^2 +\frac{\delta}{2}\|\mathbf{y}\|^2- (\mathbf{y}^T\mathbf{b}  
		- \frac{1}{2}\mathbf{x}^TH\mathbf{x} -\frac{\rho}{2}\|\mathbf{x}\|-\frac{\delta}{2}\|\mathbf{y}-\mathbf{y}_k\|)= \mathbf{x}_{\mathcal{C}}^T\mathbf{s}.
	\end{equation*}
\end{lemma}

In this work we consider an infeasible primal dual IPM for the solution of the problem \eqref{eq:QP_R_problem}, see Algorithm \ref{alg:IPM}. In particular, this is obtained considering the following Regularized Lagrangian Barrier function 

\begin{equation*} 
\begin{split}
\mathcal{L}_k(\mathbf{x}, \mathbf{y})=&\frac{1}{2}[\mathbf{x}^T, \mathbf{y}^T] \begin{bmatrix}
H+\rho I & 0 \\
0 & \delta I
\end{bmatrix}\begin{bmatrix}
\mathbf{x} \\
\mathbf{y}
\end{bmatrix} +[\mathbf{g}^T- \rho \mathbf{x}_k^T, 0 ]\begin{bmatrix}
\mathbf{x} \\
\mathbf{y}
\end{bmatrix} \\
&-\mathbf{y}^T(A
\mathbf{x} + \delta (\mathbf{y}-\mathbf{y}_k) -\mathbf{b}) - \mu \sum_{i \in \mathcal{C}} \ln (x_i)
\end{split}.
\end{equation*}
We write the corresponding KKT conditions

	\begin{align*}
	\nabla_{\mathbf{x}}\mathcal{L}_k(\mathbf{x},\mathbf{y})=	(H+ \rho I )\mathbf{x}-A^T \mathbf{y}+\mathbf{g} -\rho{\mathbf{x}_k} - \begin{bmatrix}
	0 \\
	\frac{\mu}{x_{\bar{d}+1}} \\
	\vdots \\
	\frac{\mu}{x_{d}}
	\end{bmatrix} =0 ; \label{eq:dual_R_feasibilityKKT1} \\
	-\nabla_{\mathbf{y}}\mathcal{L}_k(\mathbf{x}, \mathbf{y}) = (A\mathbf{x}+\delta (\mathbf{y}-\mathbf{y}_k) -\mathbf{b})=0 . 
	\end{align*}


Setting $s_i = \frac{\mu}{x_i}$ for $i \in \mathcal{C}$, we can then consider the following IPM map

\begin{equation}\label{eq:KKT_map}
	F_k^{\mu, \sigma}(\mathbf{x},\mathbf{y}, \mathbf{s}):=\begin{bmatrix}
	(H+ \rho I )\mathbf{x}-A^T \mathbf{y}+\mathbf{g} -\rho{\mathbf{x}_k} - \begin{bmatrix}
	0 \\
	\mathbf{s}
	\end{bmatrix} \\
A\mathbf{x}+\delta (\mathbf{y}-\mathbf{y}_k) -\mathbf{b}
	\\
	SX_{\mathcal{C}}\mathbf{e}-\sigma \mu \mathbf{e}
	\end{bmatrix}.
\end{equation}

A primal–dual interior-point method  applied to the problems \eqref{eq:QP_R_problem}-\eqref{eq:Dual_R_QP_problem}, is based on applying Newton iterations to solve a nonlinear problem of the form
\begin{equation*}
	F_k^{\mu, \sigma }(\mathbf{x},\mathbf{y}, \mathbf{s})=0, \; \; \mathbf{x}_{\mathcal{C}}>0, \; \mathbf{s}>0.
\end{equation*}
A Newton step for \eqref{eq:KKT_map} from the current iterate $(\mathbf{x},\mathbf{y}, \mathbf{s})$ is obtained by solving the system
\begin{equation*}
	\begin{bmatrix}
	 H+\rho I  &     -A^T  & \begin{bmatrix}
	 0
	 \\
	 -I
	 \end{bmatrix}
	 \\
  A & \delta I & 0 \\ 
   \begin{bmatrix}
   0 & S
   \end{bmatrix} & 0 & X_{\mathcal{C}}
	\end{bmatrix}\begin{bmatrix}
	\Delta \mathbf{x} \\
	\Delta \mathbf{y} \\
	\Delta \mathbf{s} 
	\end{bmatrix}=-F_k^{\mu,\sigma}(\mathbf{x},\mathbf{y},\mathbf{s})=:\begin{bmatrix}
	\xi_d \\
	\xi_p \\
	\xi_{\mu,\sigma}
	\end{bmatrix}.
\end{equation*}
Eliminating the variable $\Delta \mathbf{s}$ we obtain the linear system
\begin{equation}\label{eq:Newton_System}
	\underbrace{\begin{bmatrix}
	H+\rho I +\Theta^\dagger  &     -A^T 
	\\
    A & \delta I 
	\end{bmatrix}}_{\mathcal{N}_{\rho,\delta,\Theta}}\begin{bmatrix}
	\Delta \mathbf{x} \\
	\Delta \mathbf{y}  
	\end{bmatrix}=\begin{bmatrix}
	\xi_d^1 \\
	\xi^2_d+X^{-1}_{\mathcal{C}}\xi_{\mu,\sigma}\\
	\xi_p \\
	\end{bmatrix},
\end{equation}
where $\Theta^\dagger=diag([0,\dots,0];X^{-1}_{\mathcal{C}}S)$, $\xi_d^1:=[(\xi_d)_1,\dots, (\xi_d)_{\bar{d}} ]^T$ and $\xi_d^2:=[(\xi_d)_{\bar{d}+1},\dots, (\xi_d)_{{d}} ]^T$. 

In Algorithm \ref{alg:IPM} we report the IPM scheme for problem \eqref{eq:QP_R_problem}.  The method has a guaranteed polynomial convergence \cite[Chap. 6]{MR1422257} (cfr. also \cite{MR2899152,MR1242461,MR2500832,MR3082499}).
For notational simplicity we consider the case $\mathcal{C}=\{1,\dots,d\}$. To this aim, we also define
\begin{equation*}
	\begin{split}
	& \mathcal{N}_k(\bar {\gamma},\underline{\gamma},\gamma_p,\gamma_d):=\{(\mathbf{x},\mathbf{y},\textbf{s}) \;:\; \bar{\gamma} \mathbf{x}^T\mathbf{s}  \geq x_is_i \geq \underline{\gamma} \mathbf{x}^T\mathbf{s} \hbox{ for } i=1,\dots,d; \\
	& \gamma_p \mathbf{x}^T\mathbf{s} \geq \|A\mathbf{x}+\delta(\mathbf{y}-\mathbf{y}_k)-\mathbf{b}\|; \\
	& \gamma_d \mathbf{x}^T\mathbf{s} \geq \|Q\mathbf{x}+\rho(\mathbf{x}-\mathbf{x}_k)-A^T\mathbf{y}-\mathbf{s}\|\} \hbox{ and }\\
	& \begin{bmatrix}
	\mathbf{x}^j_k(\alpha)\\
	\mathbf{y}^j_{k}(\alpha) \\
	\mathbf{s}^j_{k}(\alpha) 
	\end{bmatrix}:=\begin{bmatrix}
	\mathbf{x}_{k}^j\\
	\mathbf{y}_{k}^j \\
	\mathbf{s}_{k}^j 
	\end{bmatrix}+\begin{bmatrix}
	\alpha \Delta \mathbf{x}_k^j\\
	\alpha \Delta \mathbf{y}_k^j \\
	\alpha \Delta \mathbf{s}_k^j 
	\end{bmatrix}
	\end{split}.
\end{equation*}

\begin{algorithm}[hbt!]
	\caption{Infeasible QP for problem \eqref{eq:QP_R_problem}}\label{alg:IPM}
	\KwIn{ 
		$\sigma, \bar{\sigma} \in (0,1)$ barrier reduction parameters s.t. $\sigma < \bar{\sigma}$; \\
		$\varepsilon_{p,k}>0,\varepsilon_{d,k}>0,\varepsilon_{c,k}>0$ optimality
		tolerances;}
	\init{\\
		Iteration counter $j = 0$; primal–dual point $\mathbf{x}_k^0 > 0,\;\mathbf{y}_k^0 > 0,\;\mathbf{s}_k^0 > 0$;\\
		Compute $\mu_k^0:={\mathbf{x}_k^0}^T\mathbf{s}_k^0/d$ and  $\xi^0_{d,k}$, and $\xi^0_{p,k}$.}
\While{Stopping Criterion False}{
		Solve the KKT system \eqref{eq:Newton_System} using  $[\xi^j_{p,k},\xi^j_{d,k},\xi^j_{\mu_k^j, \sigma}]^T$  to find $[\Delta \mathbf{x}_k^j,\; \Delta \mathbf{y}_k^j, \; \Delta \mathbf{s}^j_k ]^T$ \;
		Find $\alpha_k^j$ as the maximum for $\alpha \in [0,1]$ s.t. 
		$$(\mathbf{x}_k^j(\alpha),\mathbf{y}^j_k(\alpha),\mathbf{s}^j_k(\alpha) ) \in \mathcal{N}_k(\bar {\gamma},\underline{\gamma},\gamma_p,\gamma_d)$$   and 
		$$ \mathbf{x}^j_k(\alpha)^T\mathbf{s}_k^j(\alpha) \leq (1 -(1-\bar{\sigma})\alpha ){\mathbf{x}_k^j}^T\mathbf{s}_k^j \hbox{\; }$$
		
		Set $\begin{bmatrix}
		\mathbf{x}_{k}^{j+1}\\
		\mathbf{y}_{k}^{j+1} \\
		\mathbf{s}_{k}^{j+1} 
		\end{bmatrix}=\begin{bmatrix}
		\mathbf{x}_{k}^{j}\\
		\mathbf{y}_{k}^{j} \\
		\mathbf{s}_{k}^{j} 
		\end{bmatrix}+\begin{bmatrix}
		\alpha_k^j \Delta \mathbf{x}_k^j\\
		\alpha_k^j \Delta \mathbf{y}_k^j \\
		\alpha_k^j \Delta \mathbf{s}_k^j 
		\end{bmatrix}$ \;
		Compute the infeasibilities $\xi^{j+1}_{d,k}$, $\xi^{j+1}_{p,k}$ and barrier parameter $\mu^{j+1}_k:={\mathbf{x}^{j+1}_k}^T\mathbf{s}^{j+1}_k/d$ \;
		Update the iteration counter: $j := j + 1$.
	}
\end{algorithm}

\newpage

\subsection{The Proximal Stabilized-Interior Point Algorithm (PS-IPM)} 

In Algorithm \ref{alg:PS-MF-IPM} we present our proposal in full detail. 
\begin{algorithm}[hbt!]
	\caption{PS-IPM for QP}\label{alg:PS-MF-IPM}
	\KwIn{ $tol > 0$, $\sigma_r \in (0,1)$.  }
	\init{Iteration counter $k = 0$; initial point $(\mathbf{x}_0,\mathbf{y}_0)$}
	\While{$\|r(\mathbf{x}_k,\mathbf{y}_k)\|>tol$}{
		Use Algorithm \ref{alg:IPM} with starting point $(\mathbf{x}^0_{k}, \mathbf{y}^0_{k})=(\mathbf{x}_{k}, \mathbf{y}_{k})$ to find $(\mathbf{x}_{k+1},\mathbf{y}_{k+1})$ s.t. 
		\begin{equation}\label{eq:stopping_condition}
		\|r_k(\mathbf{x}_{k+1},\mathbf{y}_{k+1})\| < \frac{\min({\rho,\delta})}{\tau_1}\sigma_r^k \min\{1, \|(\mathbf{x}_{k+1}, \mathbf{y}_{k+1})-(\mathbf{x}_{k}, \mathbf{y}_{k}) \|
		\end{equation} \\
		
		Update the iteration counter: $k := k + 1$.
	}
\end{algorithm}

Two comments are in order at this stage.

\begin{enumerate}
	\item  It is important to observe that the \textit{warm starting strategy} of starting Algorithm \ref{alg:IPM} from the previous PPM approximation  $(\mathbf{x}_{k}, \mathbf{y}_{k})$ is justified by the fact that
	
	\begin{equation}\label{eq:warm_starting}
	\begin{split}
	& \|\mathcal{P}(\mathbf{x}_k) -\mathbf{x}_k \| \leq  \|\mathcal{P}(\mathbf{x}_k) -\mathcal{P}(\mathbf{x}_{k-1}) \|+ \|\mathcal{P}(\mathbf{x}_{k-1}) -\mathbf{x}_{k} \|  \\
	& \leq \eta \|\mathbf{x}_k -\mathbf{x}_{k-1} \|+ \|\mathcal{P}(\mathbf{x}_{k-1}) -\mathbf{x}_{k} \|,
	\end{split}
	\end{equation}
	where the second inequality follows from the fact that the proximal
	operator is Lipschitz continuous (see \cite[Theorem 4]{MR4047487}). Since the inexact PPM is converging we have that
	
	\begin{equation*}
	\|\mathcal{P}(\mathbf{x}_{k-1}) -\mathbf{x}_{k} \| \to 0 \hbox{ and } \|\mathbf{x}_k -\mathbf{x}_{k-1} \| \to 0,
	\end{equation*}
	proving that the proximal sub-problems  will need a non-increasing number of IPM iterations to be solved. We observe this behaviour in practice, typically after the first or second proximal iteration each subsequent proximal subproblem takes only one or two IPM iterations to converge (see Section \ref{sec:num_res_p1}) . 
	\item The IPM  Algorithm \ref{alg:IPM} uses \eqref{eq:stopping_condition} as a stopping condition.
	
\end{enumerate}

\section{Numerical Results: PS-IPM \& direct solvers} \label{sec:num_res_p1}

In this section, we present the computational results obtained by solving a set of small to
large scale linear and convex quadratic problems. We compare the performance of our proposal with that of IP-PMM \cite{MR4215213}, which, in turn, has been proven to outperform in robustness and efficiency the classic non-regularized IPM (see always \cite{MR4215213}). 
Our implementation closely follows  the one from \cite{MR4215213} and is written in Matlab\textsuperscript{\textregistered}    R2022a. 
For the solution of the (symmetrized) Newton linear systems \eqref{eq:Newton_System}, we use the  Matlab's \texttt{ldl} factorization. The factorization threshold parameter is set equal to the regularization parameter (see \eqref{eq:regularizion_par}) and is incremented by a factor $10$ if numerical instabilities are detected in the final solution of the given linear system. 

It is important to note that the presence of the regularization term stabilizes and accelerates the \texttt{ldl} routine for the Newton systems arising in Algorithm \ref{alg:IPM}, and, for this reason, we expect for our proposal similar stability and robustness properties when compared to IP-PMM. Nevertheless, from the numerical experiments presented, it will be clear that our proposal delivers a significant decrease  in the total number of IPM iterations resulting, overall, in a more efficient scheme.

Moreover, it is important to note that the reported computational times in this work are just indicative of the relative performance rather that the absolute ones. Indeed, each call of the Matlab's  \texttt{ldl} (wich uses \texttt{MA57} \cite{MR2075977}) requires an \textit{Analysis Phase} \cite[Sec. 6.2]{MR2075977}  which could be carried on just once since the sparsity pattern of the Newton matrices does not change during the IPM iterations.

Concerning the choice of the parameters in Algorithm \ref{alg:PS-MF-IPM}, we set $\sigma_r = 0.7$. Moreover, to prevent wasting time on finding excessively accurate solutions in the early PPM sub-problems, we substitute \eqref{eq:stopping_condition} with 
\begin{equation*}\label{eq:stopping_condition_empirical}
\|r_k(\mathbf{x}_{k+1},\mathbf{y}_{k+1}))\| < 10^4 \sigma_r^k \min\{1, \|(\mathbf{x}_{k+1}, \mathbf{y}_{k+1})-(\mathbf{x}_{k}, \mathbf{y}_{k}) \|.
\end{equation*}  
Indeed, in our computational experience, we found that driving the IPM solver to a high accuracy in the initial PPM iterations is unnecessary and, usually, leads to a significant deterioration of the overall performance. Concerning the initial guess, we use the same initial point as in \cite[Sec. 5.1.3]{MR4215213}, which, in turn, is based on the developments in \cite{MR1186163}. In our PS-IPM implementation, analogously to  \cite{MR4215213}, in order to find the search direction, we employ a widely used predictor-corrector method \cite{MR1186163}. This issue represents the main point where practical implementation deviates from the theory in order to gain computational efficiency. 
Concerning the stopping criterion, for the fairness of the comparison with IP-PMM, instead of using the natural residual \eqref{eq:PPM_res},
we stop the iterations of Algorithm \ref{alg:PS-MF-IPM}  when 

\begin{equation*} 
	\frac{\|\mathbf{g} -A^T\mathbf{y}+H \mathbf{x} -\mathbf{s}  \|}{\max\{\|\mathbf{g}\|,1\}} \leq \;tol \wedge \frac{\|\mathbf{b} -A\mathbf{x} \|}{\max\{\|\mathbf{b}\|,1\}} \leq \;tol \wedge \mu \leq \; tol.
\end{equation*}
Finally, we always set as regularization parameters  $\delta = \rho$, where

\begin{equation} \label{eq:regularizion_par}
	\rho = \max \{ \frac{1}{\max\{ \|A\|_{\infty},\|H\|_{\infty}\}}, 10^{-10} \}
\end{equation}
see \cite{MR4215213}. Our  large scale experiments are performed using a Dell PowerEdge R740 running Scientific Linux 7 with $4 \times$ Intel Gold 6234 3.3G, 8C/16T, 10.4GT/s, 24.75M Cache, Turbo, HT (130W) DDR4-2933.

Before showing the comparison results, we  start by briefly showcasing the theory developed until now. In particular, in Figures \ref{fig:exp1} and \ref{fig:exp2}, we report the details of the run of Algorithm \ref{alg:PS-MF-IPM} on the problems \texttt{25FV47} and \texttt{PILOT} from the Netlib collection. As the figures show, accordingly to \eqref{eq:rate_of_convergence} in Theorem \ref{th:convergence_theorem}, the rate of convergence of PPM decreases when the regularization parameter is increased (lower panels of Figures \ref{fig:exp1} and \ref{fig:exp2}). Moreover, accordingly to \eqref{eq:warm_starting}, the number of IPM iterations needed to solve the PPM sub-problems is non-increasing when the PPM iterations proceed (somehow our choice of the parameters amplifies this feature since in the majority of PPM iterations just one IPM sweep is enough to meet the inexactness criterion, see upper panels in Figures \ref{fig:exp1} and \ref{fig:exp2}).

\begin{figure}[ht!]
	\centering
	\includegraphics[width=0.48 \textwidth]{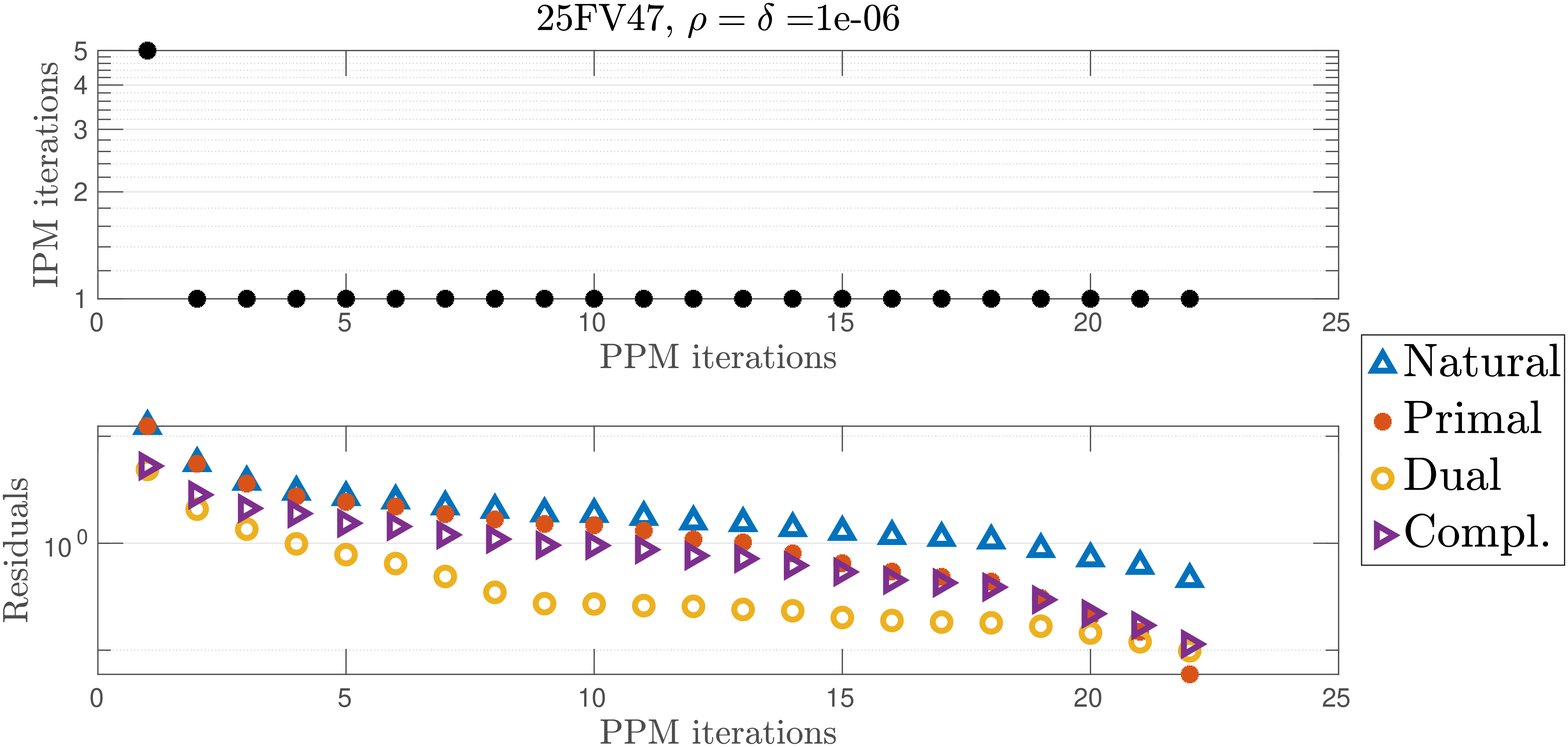} \;\includegraphics[width=0.48 \textwidth]{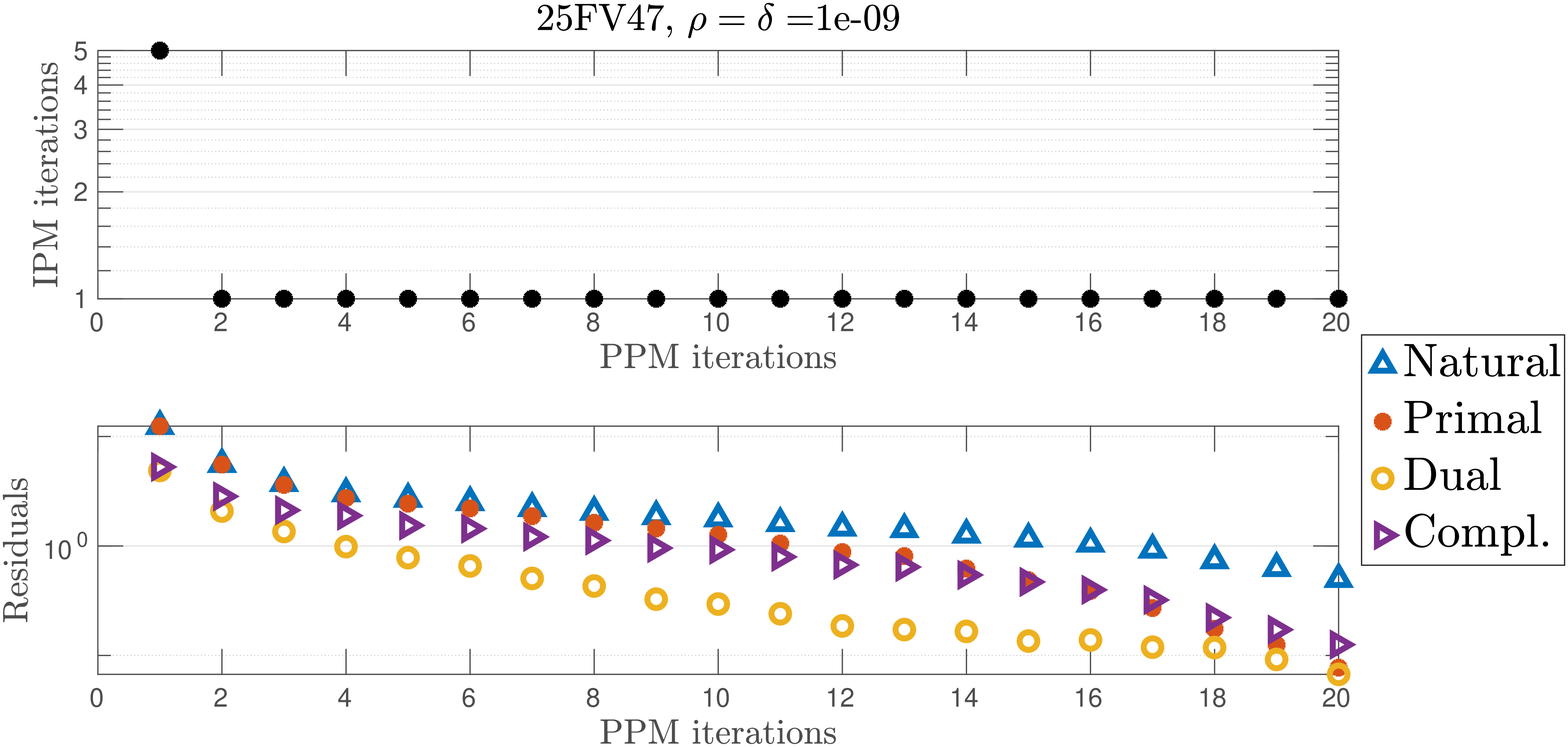}
	\caption{Problem \texttt{25FV47}. Upper Panels: PPM Iterations \& IPM Iterations. Lower Panels: Behaviour of residuals.} \label{fig:exp1}
\end{figure}

\begin{figure}[ht!]
	\centering
	\includegraphics[width=0.48 \textwidth]{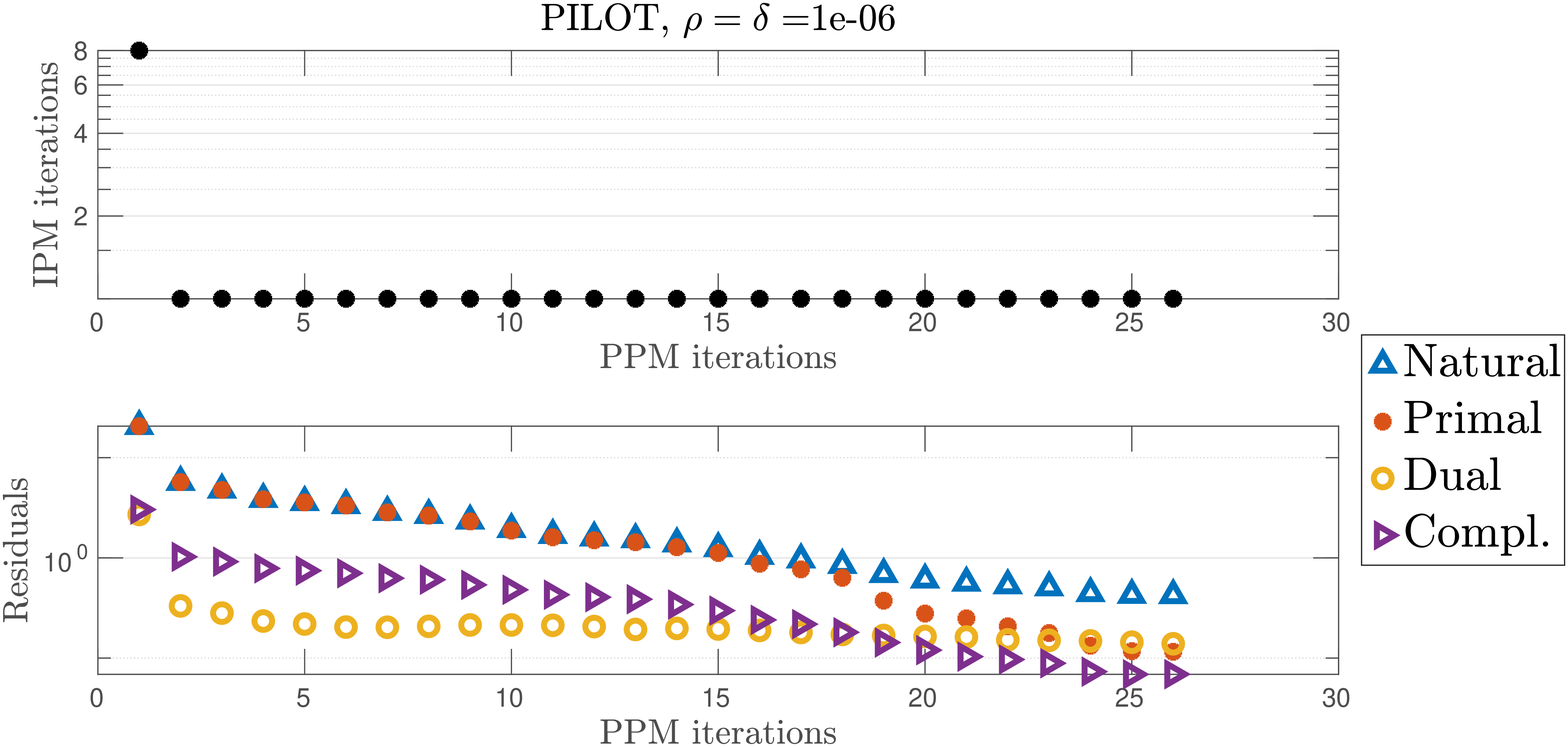} \; \includegraphics[width=0.48 \textwidth]{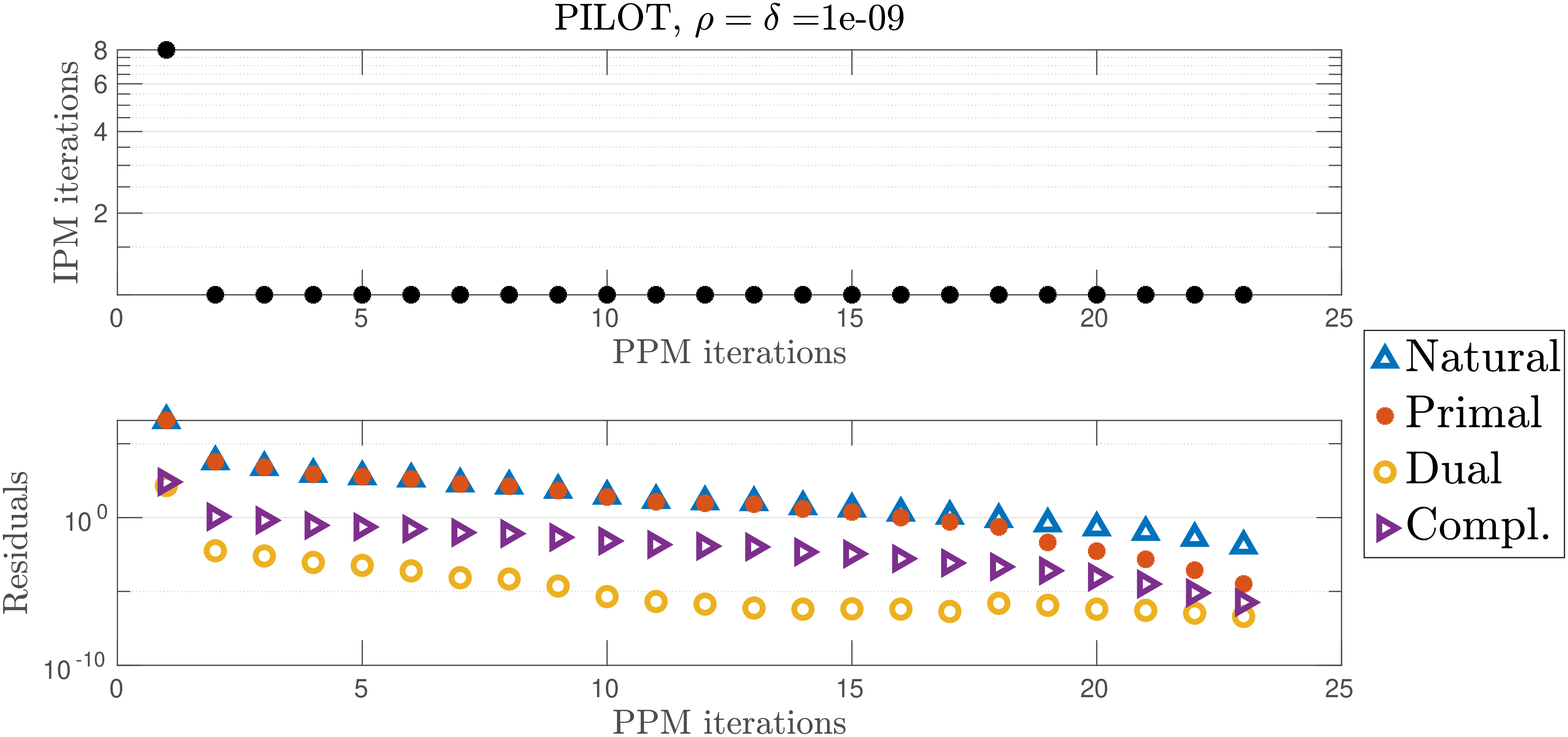}
	\caption{Problem \texttt{PILOT}. Upper panels: PPM Iterations \& IPM Iterations. Lower panels: Behaviour of residuals.} \label{fig:exp2}
\end{figure}

\subsection{Linear Programming}
The test set consists of 98 linear programming
problems from the Netlib collection. We compare the two
methods without using the pre-solved version of the problem collection (e.g. allowing rank-deficient matrices). Our proposal 
reaches the required accuracy on all the 98 problems. Hence, as expected, one of the benefits of the PPM framework becomes immediately obvious: the introduction of regularization alleviates the rank deficiency of the constraint matrix while guaranteeing convergence.
In particular, our proposal, requires a total of $1604$ PPM iterations and a total of $2518$ IPM iterations. In Figure \ref{fig:exp3} we report the performance profiles of our proposal when compared with IP-PMM \cite{MR4215213}. As revealed from the figure, the PPM framework proposed here outperforms consistently IP-PMM in terms of IPM iterations and this is reflected in a reduction of the execution time (left panel). All the obtained objective values from the two methods are comparable.

\begin{figure}[ht!]
	\centering
	\includegraphics[width=0.8 \textwidth]{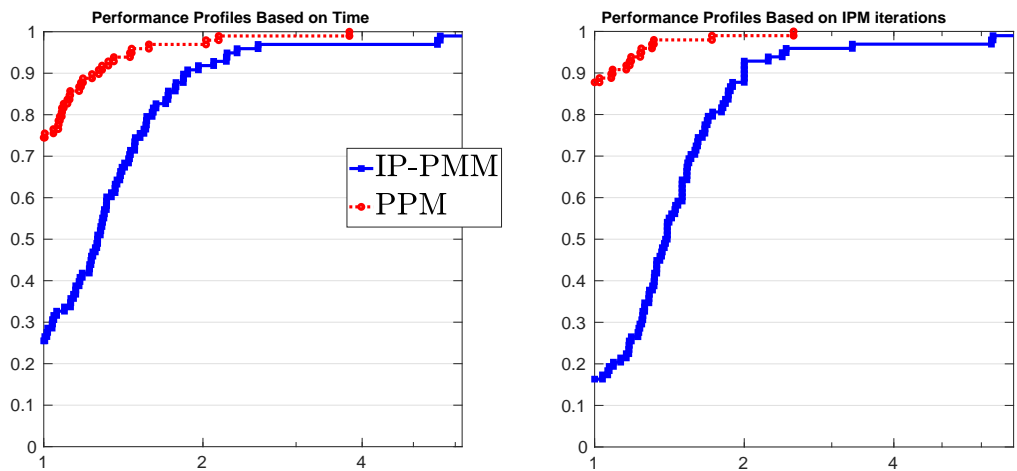} 
	\caption{Performance Profiles for Netlib's LP problems.} \label{fig:exp3}
\end{figure}

\subsection{Quadratic Programming}
Next, we present the comparison of the two methods over the Maros–M\'esz\'aros test set \cite{MR1778435}, which is comprised of 122 convex quadratic programming problems. Notice that we present the
comparison of the two methods over the set without applying any pre-processing. Our proposal reaches the required accuracy on all the 122 problems.  In particular, for our proposal the total PPM iterations were $1747$ and the total IPM iterations were $2656$. In Figure \ref{fig:exp4} we report the performance profiles for the comparison of the two methods. As it becomes apparent from the figure, the same observation which has been made for the LP case holds true also here, i.e., in the majority of cases our proposal consistently outperforms IP-PMM in terms of IPM iterations and in execution time.  All the obtained objective values from the two methods are comparable also in this case.   

\begin{figure}[ht!]
	\centering
	\includegraphics[width=0.8 \textwidth]{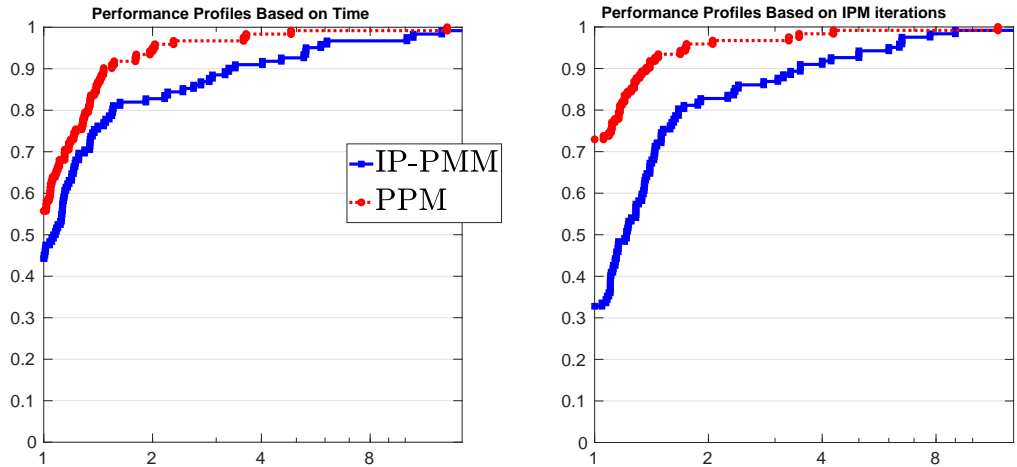} 
	\caption{Performance Profiles for Maros–M\'esz\'aros test set.} \label{fig:exp4}
\end{figure}

\subsection{Large scale problems}
All of our previous experiments were conducted on small to medium scale linear and convex quadratic programming problems. However, it is worth mentioning the limitations of the current approach, notably the memory and time required to handle the factorization. Since we employ factorizations during the iterations of the IPM, we expect that the method will be limited in terms of the size of the problems it can solve. To an extent this may be seen in Table \ref{table:1} in which we provide the statistics of the runs of the method over a set of large scale problems. It contains the number of non-zeros of the constraint matrices, as well as the time needed to solve the problem. Moreover, in Figure \ref{fig:exp5}, we report the performance profiles for the comparison of our proposal with IP-PMM: as the figures clearly highlight, our proposal outperforms IP-PMM in terms of IPM iterations and execution time further confirming the goodness of our approach. Also in this case all the obtained objective values from the two methods are comparable.

\begin{figure}[ht!]
	\centering
	\includegraphics[width=0.8 \textwidth]{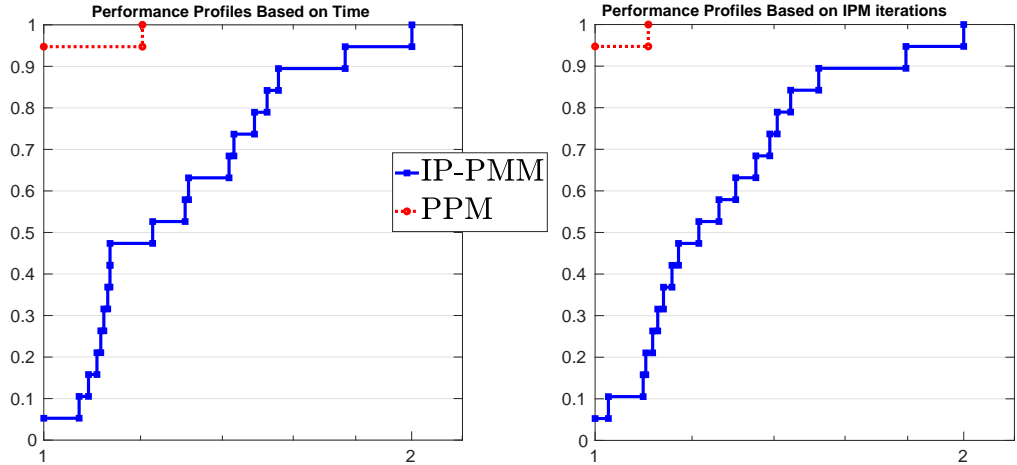} 
	\caption{Performance Profiles for Large Scale Problems.} \label{fig:exp5}
\end{figure}

Finally, in Figure \ref{fig:exp_ratio}, we report the ratio $\frac{IPM \; It.}{PPM \; It.}$. The figure further confirms the fact that thanks to the \textit{warm starting} in Algorithm \ref{alg:PS-MF-IPM} and the property \eqref{eq:warm_starting}, the average number of IPM sweeps per PPM iteration remains bounded from a worst case factor of four also for larger scale problems than those corresponding to Figures \ref{fig:exp1} and \ref{fig:exp2}.   

\newpage

\begin{table}[ht!] 
		\centering
	\scriptsize
	\caption{Details of PS-IPM performance for large scale problems} \label{table:1}
	\begin{tabular}{cccccccc}
		\topline
\headcol	{Problem} &{$nnz(A)$} &\multicolumn{1}{l}{PPM Iter} & {IPM Iter} & {Time(s)} & {Obj Val}&{Reg. Par.} & {Status} \\ \midline
		Mittelmann/fome21 & 604736 & 19 & 72 & 433.276914 & 47346318912.004189 & 5.425347e-09 & opt \\
		\rowcol
		LPnetlib/lp\_cre\_b & 260785 & 25 & 45 & 18.03 & 23129639.89 & 5.00e-09 & opt \\
		LPnetlib/lp\_cre\_d & 246614& 27 & 51 & 16.43 & 24454969.76 & 5.00e-09 & opt\\
		\rowcol
		LPnetlib/lp\_ken\_18 & 667569& 13 & 42 & 64.09 & -52217025287.38 & 5.00e-09 & opt\\
		Qaplib/lp\_nug20 &304800 & 17 & 17 & 258.97 & 2181.63 & 1.25e-07 & opt \\
		\rowcol
		LPnetlib/lp\_osa\_30 &604488 & 20 & 34 & 14.40 & 2142139.89 & 5.00e-09 & opt  \\
		LPnetlib/lp\_osa\_60 &1408073 & 18 & 33 & 42.82 & 4044072.58 & 5.00e-09 & opt\\
		\rowcol
		LPnetlib/lp\_pds\_10 & 139901& 20 & 44 & 24.00 & 26727094976.00 & 5.42e-09 & opt \\
		LPnetlib/lp\_pds\_20 & 302423& 19 & 61 & 209.19 & 23821658640.00 & 5.42e-09 & opt \\
		\rowcol
		LPnetlib/lp\_stocfor3 &72721 & 30 & 53 & 3.36 & -39976.78 & 5.00e-09 & opt\\ 
		Mittelmann/pds-100 &1515296 & 20 & 81 & 4362.05 & 10928229968.05 & 5.00e-09 & opt \\
		\rowcol
		Mittelmann/pds-30 & 447659& 23 & 70 & 467.42 & 21385445736.00 & 5.42e-09 & opt\\
		Mittelmann/pds-40 & 617606& 19 & 71 & 1066.96 & 18855198824.11 & 5.42e-09 & opt \\
		\rowcol
		Mittelmann/pds-50 &787867 & 20 & 72 & 1234.02 & 16603525724.00 & 5.42e-09 & opt \\
		Mittelmann/pds-60 &965265 & 19 & 74 & 1791.66 & 14265904407.18 & 5.42e-09 & opt \\
		\rowcol
		Mittelmann/pds-70 &1126605& 19 & 79 & 2704.41 & 12241162812.00 & 5.42e-09 & opt \\
		Mittelmann/rail2586 & 8011362 & 33 & 76 & 2371.13& 936.66 & 5.00e-09 & opt \\
		\rowcol
		Mittelmann/rail4284 & 11284032& 34 & 65 & 2720.71 & 1054.51 & 5.00e-09 & opt\\
		Mittelmann/rail582 & 402290 & 33 & 35 & 71.28 & 209.72 & 5.00e-09 & opt \\ 
		\bottomlinec                    
	\end{tabular}
\end{table}

\begin{figure}[htb!]
	\centering
	\includegraphics[width=0.5 \textwidth]{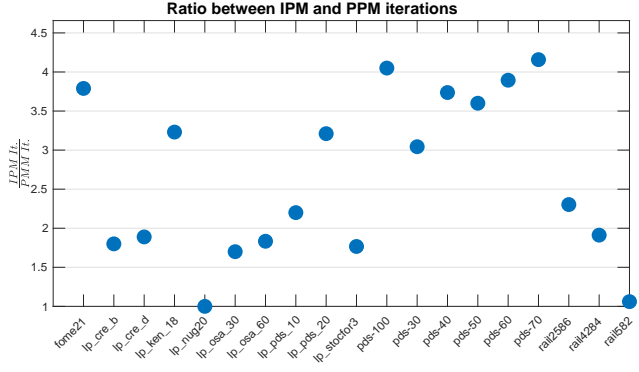}
	\caption{Average IPM sweeps per PPM iteration} \label{fig:exp_ratio}
\end{figure}

\section{Slack formulation and preconditioning} \label{sec:slac_prec}
The presence of proximal point regularization brings 
several advantages to the IPM. One of them is bounding 
the spectrum of the matrices in Newton system \cite{MR2891903, MR3010439,MR3682906}. 
In this section, we will show how this may be combined 
with a trick of replicating the variables involved 
in the inequality constraints to deliver a completely 
new and much desirable feature and alleviating 
the inherent numerical instability which originates 
from the IPM scaling matrix $\Theta$.   

For the sake of simplicity, in this section, we assume that in \eqref{eq:QP_problem} all variables have nonnegative constraints. 
Then, using the trick of variable replication we get the following slack formulation \cite[Sec. 6]{MR2899152} of the original problem \eqref{eq:QP_problem}:

\begin{align} \label{eq:QP_problem_slack}
\begin{aligned}
\min_{\mathbf{x} \in \mathbb{R}^{d_1}} \;& \frac{1}{2}\mathbf{x}^TH\mathbf{x} +\mathbf{g}^T\mathbf{x}  \\ 
\hbox{s.t.} \; & A\mathbf{x}= \mathbf{b}, \; \mathbf{x}- \mathbf{z} =0  \\
& \mathbf{z}  \geq 0 \\
\end{aligned}
\end{align}
%
In this case the IPM map \eqref{eq:KKT_map} can be written as

\begin{equation}\label{eq:KKT_map_slack}
F_k^{\mu, \sigma}(\mathbf{x},\mathbf{z},\mathbf{y}_1,\mathbf{y}_2, \mathbf{s}):=\begin{bmatrix}
\begin{bmatrix}
H + \rho I  & 0 \\
0           & \rho I
\end{bmatrix}\begin{bmatrix}
\mathbf{x} \\
\mathbf{z}
\end{bmatrix}+\begin{bmatrix}
\mathbf{g}-\rho \mathbf{x}_k \\
- \rho \mathbf{z}_k
\end{bmatrix} - \begin{bmatrix}
A^T & I \\
0 & -I
\end{bmatrix}\begin{bmatrix}
\mathbf{y}_1 \\
\mathbf{y}_2
\end{bmatrix}- \begin{bmatrix}
0 \\
\mathbf{s}
\end{bmatrix} \\

\begin{bmatrix}
A & 0 \\
I & -I
\end{bmatrix}\begin{bmatrix}
\mathbf{x} \\
\mathbf{z}
\end{bmatrix}+\delta \begin{bmatrix}
\mathbf{y}^1 -\mathbf{y}^1_{k}\\
\mathbf{y}^2 -\mathbf{y}^2_{k}
\end{bmatrix} - \begin{bmatrix}
\mathbf{b} \\
0
\end{bmatrix}
\\
SZ\mathbf{e}-\sigma \mu \mathbf{e}
\end{bmatrix}.
\end{equation}
Using \eqref{eq:KKT_map_slack}, the corresponding Newton system (see also equation \eqref{eq:Newton_System}), can be expressed as
\begin{equation}\label{eq:Newton_System_decoupled}
{\begin{bmatrix}
	H+\rho I  & 0                    &  -A^T    & -I \\
	0         & \Theta^{-1} +\rho I  &   0      & I \\
	A         &       0              &   \delta I & 0 \\ 
	I         &      -I              &     0      & \delta I \\ 
	\end{bmatrix}}\begin{bmatrix}
\Delta \mathbf{x} \\
\Delta \mathbf{z} \\
\Delta \mathbf{y}_1 \\
\Delta \mathbf{y}_2
\end{bmatrix}=\begin{bmatrix}
\xi_d^1 \\
\xi^2_d+Z^{-1}\xi_{\mu,\sigma}\\
\xi^1_p \\
\xi^2_p
\end{bmatrix}, \hbox{ where } \Theta = ZS^{-1}. 
\end{equation}

For the convenience of the reader we also report, in the following, the explicit expressions of the IPM residuals: the natural PPM residual in \eqref{eq:PPM_res} reads as

\begin{equation*} 
r(\mathbf{x},\mathbf{z}, \mathbf{y}):= \begin{bmatrix}
H\mathbf{x}+\mathbf{g}-[A^T\; \; I ]\mathbf{y} \\
\mathbf{z} - \Pi_{\mathbb{R}_{\geq 0}}(\mathbf{z}-(-[0\; \; -I]\mathbf{y})) \\
A\mathbf{x}-\mathbf{b} \\
\mathbf{x}-\mathbf{z} 
\end{bmatrix},
\end{equation*} 
whereas the residual in \eqref{eq:PPM_res_reg} becomes

\begin{equation*} 
r_k(\mathbf{x},\mathbf{z},\mathbf{y}):= \begin{bmatrix}
H\mathbf{x}+\mathbf{g}-[A^T \; \; I ]\mathbf{y} +\rho (\mathbf{x}-\mathbf{x}_k)  \\
\mathbf{z}- \Pi_{\mathbb{R}_{\geq 0}}(\mathbf{z} - (\rho(\mathbf{z}-\mathbf{z}_k)- [0\; \; -I]\mathbf{y}))\\
A\mathbf{x}-\mathbf{b}+ \delta (\mathbf{y}^1-\mathbf{y}^1_k) \\
\mathbf{x}-\mathbf{z} +\delta (\mathbf{y}^2-\mathbf{y}^2_k)
\end{bmatrix}.
\end{equation*}

\subsection{Solution of the Newton system} \label{sec:preconditioning}

In this section we will study in details the solution of the linear system \eqref{eq:Newton_System_decoupled} when reordered and symmetrized as follows:

\begin{equation}\label{eq:Newton_System_decoupled_permuted}
\underbrace{\begin{bmatrix}
	\Theta^{-1} +\rho I              & -I        &   0       & 0 \\
	-I                               & -\delta I &   I       & 0 \\
	0                                &  I        &   H+\rho I& A^T \\
	0                                &  0        &   A       & -\delta I\\   
	\end{bmatrix}}_{=:\mathcal{N}(\Theta)}\begin{bmatrix}
\Delta \mathbf{z} \\
-\Delta \mathbf{y}_2 \\
\Delta \mathbf{x} \\
-\Delta \mathbf{y}_1

\end{bmatrix}=\begin{bmatrix}
\xi^2_d+Z^{-1}\xi_{\mu,\sigma}\\
\xi^2_p \\
\xi_d^1 \\
\xi^1_p
\end{bmatrix}.
\end{equation}
To this aim, let us partition the matrix $\mathcal{N}(\Theta)$ as 

\begin{equation*}
	\mathcal{N}(\Theta)=\begin{bmatrix}
	N_{11}(\Theta) & N_{12} \\
	N_{21} & N_{22}
	\end{bmatrix},
\end{equation*}  
where \begin{equation*}
N_{11}(\Theta):=\begin{bmatrix}
\Theta^{-1} +\rho I              & -I    \\
-I                               & -\delta I
\end{bmatrix}, \; N_{12}:=\begin{bmatrix}
0 & 0 \\
I & 0
\end{bmatrix}, \;  N_{21}=N_{12}^T, \;  N_{22}:= \begin{bmatrix}
H+\rho I & A^T \\
A       & -\delta I
\end{bmatrix}.
\end{equation*}

\noindent Before continuing, let us observe that, under suitable hypothesis, the solution of a linear system of the form

\begin{equation*}
\begin{bmatrix}
G_{11}  & G_{12} \\
G_{21}       & G_{22}
\end{bmatrix}\begin{bmatrix}
\mathbf{x} \\
\mathbf{y} 
\end{bmatrix}=\begin{bmatrix}
\mathbf{b}_\mathbf{x} \\
\mathbf{b}_\mathbf{y} 
\end{bmatrix},
\end{equation*} 
can be obtained solving 
\begin{equation} \label{eq:block_solution}
\begin{cases}
(G_{22}-G_{21}G_{11}^{-1}G_{12})\mathbf{y}=\mathbf{b}_{\mathbf{y}}-G_{21}G_{11}^{-1}\mathbf{b}_{\mathbf{x}} \\
\mathbf{x}=G_{11}^{-1}(\mathbf{b}_{\mathbf{x}} - G_{12}\mathbf{y} ).

\end{cases}
\end{equation}

Since the linear systems involving $N_{11}(\Theta)$ are easily solvable, using \eqref{eq:block_solution}, the overall solution of the linear system \eqref{eq:Newton_System_decoupled_permuted} can be  obtained from the following two ancillary ones:

\begin{subequations}\label{eq:Newton_Solution}
	\begin{align}
	 S(\Theta)\begin{bmatrix}
	 \Delta \mathbf{x} \\
	 -\Delta \mathbf{y}_1
	 \end{bmatrix} =(\begin{bmatrix}
	 \xi_d^1 \\
	 \xi^1_p
	 \end{bmatrix}-N_{21}N_{11}(\Theta)^{-1}\begin{bmatrix}
	 \xi^2_d+Z^{-1}\xi_{\mu,\sigma}\\
	 \xi^2_p 
	 \end{bmatrix}
	 ) \label{eq:Newton_Solution_sys1} \\
	 N_{11}(\Theta)\begin{bmatrix}
	 \Delta \mathbf{z} \\
	 -\Delta \mathbf{y}_2
	 \end{bmatrix}=(\begin{bmatrix}
	 	\xi^2_d+Z^{-1}\xi_{\mu,\sigma}\\
	 	\xi^2_p 
	 \end{bmatrix}- N_{12}\begin{bmatrix}
	 \Delta \mathbf{x} \\
	 -\Delta \mathbf{y}_1
	 \end{bmatrix}), \label{eq:Newton_Solution_sys2}
	\end{align}
\end{subequations}
where $S(\Theta)$ is the Schur complement

\begin{equation} \label{eq:Scur_Separation}
S(\Theta):= \begin{bmatrix}
H+\rho I + (\delta I +(\Theta^{-1}+ \rho I)^{-1})^{-1}  & A^T \\
A       & -\delta I
\end{bmatrix}.
\end{equation}

It is important to observe that, at this stage, the reasons to go through  the current reformulation of the problem are not completely apparent: we essentially doubled the dimension of the primal variables ending up with the necessity of solving linear systems
involving a Schur complement, see equation  \eqref{eq:Scur_Separation}, which has exactly the same sparsity pattern as the standard (symmetrized) Newton system

\begin{equation} \label{eq:stadard_newton}
\mathcal{N}_{C}(\Theta):=\begin{bmatrix}
	H+\rho I +\Theta^{-1}  &     A^T 
	\\
	A & -\delta I 
	\end{bmatrix},
\end{equation}
cfr. equation \eqref{eq:Newton_System}.  

In the following Remarks \ref{rem:1} and \ref{rem:2} we highlight the advantages given by the current reformulation of the Newton system showing, in essence, that the formulation  in \eqref{eq:Scur_Separation} allows \textit{better preconditioner re-usage}  than in the standard formulation \eqref{eq:stadard_newton}. To this aim, as it is customary in IPM methods, let us suppose that

\begin{equation*}
 \lambda_{max}(\Theta^{-1})=O(\frac{1}{\mu})  \hbox{ and } \lambda_{min}(\Theta^{-1})=O({\mu}),
\end{equation*}
where $\mu$ is the average complementarity product at any given IPM iteration. Using the above assumption, we obtain
\begin{equation}\label{eq:cluster_limit}
\begin{split}
& \lim_{\mu \to 0} \lambda_{min}(\delta I + (\Theta^{-1}+\rho I )^{-1})^{-1} = \frac{\rho}{\delta \rho +1} \\
& \lim_{\mu \to 0} \lambda_{max}(\delta I + (\Theta^{-1}+\rho I )^{-1})^{-1} = \frac{1}{\delta}.
\end{split}
\end{equation} 

From equation \eqref{eq:cluster_limit} the main advantage of dealing with the formulation \eqref{eq:Scur_Separation} of the Schur complement
becomes more apparent: whilst the elements of the diagonal IPM matrix appearing in $\mathcal{N}_C(\Theta)$ are such that $\Theta^{-1}_{ii} \in (0,+\infty)$ when $\mu \to 0$,  the diagonal elements appearing in the Schur complement \eqref{eq:Scur_Separation} belong to the interval $(\frac{\rho}{\delta \rho +1}, \frac{1}{\delta})$ when $\mu \to 0$.

\begin{remark} \label{rem:1}
	When $\mu \to 0$, for the variables which have been identified as active or inactive by the PS-IPM, see Algorithm \ref{alg:PS-MF-IPM}, we have that 
	\begin{equation*}
	(\delta I + (\Theta^{-1}+\rho I )^{-1})^{-1}_{ii}\approx \frac{1}{\delta} \quad \hbox{ or } \quad (\delta I + (\Theta^{-1}+\rho I )^{-1})^{-1}_{ii}\approx \frac{\rho}{\delta \rho +1},
	\end{equation*} 
	and such values are expected to remain unchanged in the following PS-IPM steps. This suggests that, when  close enough to convergence, any computed approximation of the matrix $S(\Theta)$ may be used as an effective preconditioner  also for subsequent PS-IPM steps.
\end{remark}

In Lemma \ref{lem:dumping_factors} we show that the regularization parameters $(\rho,\delta)$ act as \textit{dumping coefficients}, see \eqref{eq:diagonal_variation}, for the variations $\widehat{\Theta}^{-1}_{ii} -\Theta^{-1}_{ii}$
where $\widehat{\Theta}^{-1}$ and $\Theta^{-1}$ are two IPM matrices obtained, respectively, in two different IPM iterations.	

\begin{lemma}\label{lem:dumping_factors}
	Define 
	$$D_A:=(\delta I + (\widehat{\Theta}^{-1}+\rho I )^{-1})^{-1} -(\delta I + (\Theta^{-1}+\rho I )^{-1})^{-1}$$
	and 
	$$D_C:=\widehat{\Theta}^{-1} -\Theta^{-1}.$$
	Then,
	\begin{equation}\label{eq:approx_ineq}
	\|{S}(\widehat{\Theta})-S(\Theta)\|_2=\|D_A\|_2<\|D_C\|_2=\|{\mathcal{N}}_{C}(\widehat{\Theta})-{\mathcal{N}}_{C}(\Theta)\|_2.
	\end{equation}
\end{lemma}
\begin{proof}
	From direct computation we have
	\begin{equation}\label{eq:diagonal_variation}
	(D_A)_{ii}=\frac{\widehat{\Theta}_{ii}^{-1} -\Theta_{ii}^{-1} }{1 + \delta^2 ({\Theta}_{ii}^{-1}+\rho)(\widehat{\Theta}_{ii}^{-1}+\rho)+\delta ({\Theta}_{ii}^{-1}+\rho)+\delta(\widehat{\Theta}_{ii}^{-1}+\rho)}.
	\end{equation}
	Then,
	\begin{equation*}
	|(D_A)_{ii}|<|(D_C)_{ii}|,
	\end{equation*}
	and the thesis follows using the definitions of $S(\Theta)$ and $\mathcal{N}_C(\Theta)$.
\end{proof}

\begin{remark}\label{rem:2} 
	Suppose we computed a preconditioner for ${S}(\Theta)$, e.g., an incomplete factorization. Equation \eqref{eq:approx_ineq} shows that any accurate preconditioner for $S(\Theta)$ approximates ${S}(\widehat{\Theta})$ better than an analogous preconditioner 
	for ${\mathcal{N}}_{C}(\Theta)$ would approximate ${\mathcal{N}}_{C}(\widehat{\Theta})$. 
	
	Moreover, from \eqref{eq:diagonal_variation}, we can observe that the variations $(D_{A})_{ii}$, and hence the variation $\|{S}(\widehat{\Theta})-S(\Theta)\|_2$ of the overall Schur complement,  are negatively correlated with the regularization parameters $(\rho, \delta)$. Then, it has to be expected that the computed preconditioner for $S(\Theta)$ would be yet an effective preconditioner for the matrix ${S}(\widehat{\Theta})$ if the regularization parameters $(\rho, \delta)$ are sufficiently large. 
	
	On the other hand, according to \eqref{eq:rate_of_convergence},  the rate of convergence of PPM correlates inversely  with the regularization parameters $(\rho, \delta)$. 
	
	As a result of the above discussion, we are able to unveil a \textit{precise interaction} between the computational footprint related to the necessity of re-computing preconditioners and the rate of convergence of the PS-IPM with the obvious benefit to allow  a predictable tuning of such trade-off (see Section \ref{sec:num_res_p2}). 
	\end{remark}
To conclude this section, in Theorem \ref{lem:clustering_eig}, we analyse in more detail the eigenvalues of the matrix ${S}({\Theta})^{-1}{S}(\widehat{\Theta})$. Indeed, supposing we have computed an accurate preconditioner for ${S}({\Theta})$, then, the eigenvalues of ${S}({\Theta})^{-1}{S}(\widehat{\Theta})$ may be considered as a measure of the effectiveness of such preconditioner when used as preconditioner for ${S}(\widehat{\Theta})$: the results there contained will further confirm that a high quality preconditioner is expected when the matrix $D_A$ has small diagonal elements, see \eqref{eq:cluster_interval}. In this case, indeed  ${S}({\Theta})^{-1}{S}(\widehat{\Theta})$ has a highly clustered spectrum.

\begin{theorem}\label{lem:clustering_eig}
	Let us define $$H_{A,\Theta, \delta, \rho}:=H+\rho I+\frac{1}{\delta}A^TA+ (\delta I+({\Theta}^{-1}+ \rho I)^{-1})^{-1}.$$ Then, the matrix ${S}({\Theta})^{-1}{S}(\widehat{\Theta})$ has the eigenvalue $\eta=1$ with multiplicity at least $m_1$ whereas, the other eigenvalues, are s.t.
	{\begin{equation} \label{eq:cluster_interval}
		\eta \in (1+ \frac{\min_{i}\lambda_i(D_A) }{\max_{i}\lambda_i(H_{A,\Theta, \delta, \rho})}, 1+\frac{\max_{i}\lambda_i(D_A) }{\min_{i}\lambda_i(H_{A,\Theta, \delta, \rho})})  ).
		\end{equation}}
\end{theorem}
\begin{proof}
	To analyse the eigenvalues we consider the problem
	\begin{equation*}
	S(\widehat{\Theta})\mathbf{u}= \eta S({\Theta}) \mathbf{u}, 
	\end{equation*}
	i.e.,
	\begin{equation} \label{eq:gener_eig}
	\begin{bmatrix}
	H+\rho I+(\delta I +(\widehat{\Theta}^{-1}+ \rho I)^{-1})^{-1}& A^T \\
	A       & -\delta I
	\end{bmatrix}\begin{bmatrix}
	\mathbf{u}_1 \\
	\mathbf{u}_2
	\end{bmatrix}= \eta \begin{bmatrix}
	H+\rho I+(\delta I+({\Theta}^{-1}+ \rho I)^{-1})^{-1}& A^T \\
	A       & -\delta I
	\end{bmatrix} \begin{bmatrix}
	\mathbf{u}_1 \\
	\mathbf{u}_2
	\end{bmatrix}. 
	\end{equation}
	
	If $\eta=1$, we obtain that any vector of the form $[0,\mathbf{u}_2]$ is a solution of \eqref{eq:gener_eig} and hence the multiplicity of eigenvalue $1$ is at least $m_1$. Let us suppose $\eta \neq 1$. Always from \eqref{eq:gener_eig} we obtain $A\mathbf{u}_1=\delta \mathbf{u}_2$ and hence, using the equality
	
	\begin{equation*}
		(H+\rho I+(\delta I +(\widehat{\Theta}^{-1}+ \rho I)^{-1})^{-1}+\frac{1}{\delta}A^TA)\mathbf{u}_1=\eta(H+\rho I+(\delta I +({\Theta}^{-1}+ \rho I)^{-1})^{-1}+\frac{1}{\delta}A^TA)\mathbf{u}_1, 
	\end{equation*}
	we obtain
	\begin{equation}\label{eq:eqig_expr}
	\eta = 1 + \frac{\mathbf{u}_1^T((\delta I+(\widehat{\Theta}^{-1}+ \rho I)^{-1})^{-1} -(\delta I+({\Theta}^{-1}+ \rho I)^{-1})^{-1})\mathbf{u}_1}{\mathbf{u}_1^T((H+\rho I+\frac{1}{\delta}A^TA+ (\delta I+({\Theta}^{-1}+ \rho I)^{-1})^{-1})\mathbf{u}_1 }.
	\end{equation}
	Thesis follows from \eqref{eq:eqig_expr} using the definition of $D_A$, $H_{A,\Theta, \delta, \rho}$ and the fact that both are  symmetric matrices. 
	\end{proof}

\subsubsection{Further Schur complement reduction} \label{sec:further_schur}
In some particular cases, using once more \eqref{eq:block_solution}, it might be computationally advantageous to further reduce the solution of the linear system in \eqref{eq:Newton_Solution_sys1} to a smaller linear system involving its Schur complement. Among other situations, this is the case of IPM matrices coming from problems of the form \eqref{eq:QP_problem_slack} for which $H$ is diagonal or where $H=0$, see, e.g., the LP case.  In this section we prove a similar result to Lemma~\ref{lem:dumping_factors}  when the involved matrices are the Schur complements of the linear systems \eqref{eq:Scur_Separation} and \eqref{eq:stadard_newton}. To this aim  and for the sake of simplicity, we consider the case $H=0$ and define $L_1(\Theta)$ as the Schur complement of \eqref{eq:Scur_Separation}, i.e.,

\begin{equation} \label{eq:Schur_further}
L_1(\Theta):=-(\delta I +A(\rho I +(\delta I +(\Theta^{-1}+\rho I)^{-1})^{-1})^{-1}A^T),
\end{equation}
whereas we define $L_2(\Theta)$ as the Schur complement of \eqref{eq:stadard_newton}, i.e.,
\begin{equation*}
L_2(\Theta):=-(\delta I +A(\rho I + \Theta^{-1})^{-1}A^T).
\end{equation*}

Moreover, considering diagonal scaling matrices 
$\widehat{\Theta}^{-1}$   and   $\Theta^{-1}$  obtained, 
respectively in two different IPM iterations, 
we have 

\begin{equation*}
{L}_1(\widehat{\Theta})-L_1(\Theta) = A[\underbrace{(\rho I +(\delta I +(\rho I+\Theta^{-1} )^{-1})^{-1})^{-1}-(\rho I +(\delta I +(\rho I+\widehat{\Theta}^{-1} )^{-1})^{-1})^{-1}}_{=:\Delta_{1,\widehat{\Theta},\Theta}}]A^T
\end{equation*}
whereas
\begin{equation*}
{L}_2(\widehat{\Theta})-L_2(\Theta) = A[\underbrace{(\rho I+\Theta^{-1} )^{-1})-(\rho I+\widehat{\Theta}^{-1} )^{-1}}_{=:\Delta_{2,\widehat{\Theta},\Theta}}]A^T.
\end{equation*}

We are ready to state Lemma \ref{lem:dunmpig_schur_schur}, which guarantees that also when operating a further reduction to the Schur complement in \eqref{eq:Scur_Separation}, the regularization parameters $(\rho,\delta)$ act as \textit{dumping factors} for the changes in the diagonal matrix $|\Delta_{2,\widehat{\Theta},\Theta}|$. 

\begin{lemma} \label{lem:dunmpig_schur_schur}
	With the notation introduced above,  we have
	
	\begin{equation}\label{eq:further_Schur_diagonal_variation}
		|(\Delta_{1,\widehat{\Theta},\Theta})_{ii}|<|(\Delta_{2,\widehat{\Theta},\Theta})_{ii}|
	\end{equation}  	
	
	and
	
	\begin{equation*} \label{eq:approx_ineq_schur}
	\|A|\Delta_{1,\widehat{\Theta},\Theta}|A^T\|_2<\|A|\Delta_{2,\widehat{\Theta},\Theta}|A^T\|_2 . 
	\end{equation*}
\end{lemma}
\begin{proof}

From direct computation, we have that

\begin{equation*}\label{eq:Schur_Diag_ineq}
	\begin{split}
	(\Delta_{1,\widehat{\Theta},\Theta})_{ii}=&\frac{\widehat{\Theta}_{ii}^{-1} - {\Theta}_{ii}^{-1}}{(\rho(\delta + (\widehat{\Theta}^{-1}+\rho I)_{ii}^{-1})+1) (\rho(\delta +({\Theta}^{-1}+\rho I)_{ii}^{-1})+1)(\widehat{\Theta}^{-1}_{ii}+\rho)({\Theta}^{-1}_{ii}+\rho)} = \\
	& \frac{1}{(\rho(\delta + (\widehat{\Theta}^{-1}+\rho I)_{ii}^{-1})+1) (\rho(\delta +({\Theta}^{-1}+\rho I)_{ii}^{-1})+1)}\frac{\widehat{\Theta}_{ii}^{-1} - {\Theta}_{ii}^{-1}}{(\widehat{\Theta}^{-1}_{ii}+\rho)({\Theta}^{-1}_{ii}+\rho)},
	\end{split}
\end{equation*}
from which, observing that $(\Delta_{2,\widehat{\Theta},\Theta})_{ii}=\frac{\widehat{\Theta}_{ii}^{-1} - {\Theta}_{ii}^{-1}}{(\widehat{\Theta}^{-1}_{ii}+\rho)({\Theta}^{-1}_{ii}+\rho)}$, follows \eqref{eq:further_Schur_diagonal_variation}. The second part of the statement follows observing that

\begin{equation*}
\begin{split}
\|A|\Delta_{1,\widehat{\Theta},\Theta}|A^T\|_2 = & \max_{\mathbf{x}} \frac{\mathbf{x}^TA|\Delta_{1,\widehat{\Theta},\Theta}|A^T\mathbf{x}}{\mathbf{x}^T\mathbf{x}} =  \\
& \max_{\mathbf{x}} \frac{\mathbf{x}^TA|\Delta_{1,\widehat{\Theta},\Theta}|A^T\mathbf{x}}{\mathbf{x}^TAA^T\mathbf{x}}\frac{\mathbf{x}^TAA^T\mathbf{x}}{\mathbf{x}^T\mathbf{x}} < \\
& \max_{\mathbf{x}} \frac{\mathbf{x}^TA|\Delta_{2,\widehat{\Theta},\Theta}|A^T\mathbf{x}}{\mathbf{x}^TAA^T\mathbf{x}}\frac{\mathbf{x}^TAA^T\mathbf{x}}{\mathbf{x}^T\mathbf{x}} =\|A|\Delta_{2,\widehat{\Theta},\Theta}|A^T\|_2, 
\end{split}
\end{equation*}
where in the last inequality we used \eqref{eq:further_Schur_diagonal_variation}.

\end{proof}

\section{Numerical Results: PS-IPM \& iterative solvers} \label{sec:num_res_p2}
In this section we present the computational results obtained using Algorithm \ref{alg:PS-MF-IPM} when the problem is reformulated as  in  \eqref{eq:QP_problem_slack} and the corresponding linear systems arising from the PS-IPM subproblems are solved using \eqref{eq:Newton_Solution}.  In particular, in the first part of this section we consider the case when the linear system \eqref{eq:Newton_Solution_sys1} is solved 
without further reduction to Schur complement, whereas, in the second one, we present the numerical results for this case (see also Section \ref{sec:further_schur}).

\subsection{GMRES+\texttt{ldl}} \label{sec:gmres_ldl}
In the first case, for the solution of the linear system \eqref{eq:Newton_Solution_sys1}, we use  \texttt{GMRES(100,1)} \cite{MR848568}. Moreover, as suggested in the discussion in Section \ref{sec:preconditioning}, as preconditioner of a given Schur complement $S(\widehat{\Theta})$, we use the \texttt{ldl} decomposition of $S({\Theta})$ computed in a previous PS-IPM iteration. It is important to note that when GMRES is applied to a non-normal matrix its convergence behaviour is not fully determined by its spectral distribution, or better, its spectral distribution is completely irrelevant \cite{MR1397238}.
Nevertheless, when $\|S(\Theta) -  S(\widehat{\Theta})\|_2 $ is small, we expect the matrix ${S}({\Theta})^{-1}{S}(\widehat{\Theta})$ to be \textit{close} to the identity. In this case, and in general for symmetric matrices, the spectral distribution is fully  descriptive of GMRES behaviour \cite[Cor. 6.33]{MR1990645} and hence we expect {GMRES} to behave accordingly to the spectrum of the preconditioned matrix as in Theorem \ref{lem:clustering_eig}. As previously mentioned, we factorize the Schur complement in \eqref{eq:Scur_Separation} using Matlab's \texttt{ldl} routine  and, in our experiments, this factorization is recomputed if in the current PS-IPM step, GMRES has performed more than the $51\%$ of the maximum allowed iterations in the solution of at least one of the two predictor-corrector systems. The stopping (absolute) tolerance for GMRES is set as $\min(10^{-1},0.8\mu)$ where $\mu$ is the current duality gap.  The other computational details are analogous to those used in Section \ref{sec:num_res_p1}.

All the computational results presented here  are devoted to showcase the theory developed in Section \ref{sec:preconditioning} and, in particular, to show that our proposal needs, in general, a number of factorizations equal to a fraction of the IPM iterations, delivering considerable savings of computational time for problems where the factorization footprint is dominant. 

In the following discussion we will use the ratio $\frac{IPM It.}{Fact.}$ as a measure of the frequency at which the preconditioner is recomputed. Moreover, we will use the ratio $\frac{Kryl. It.}{Fact.}$ as a measure of how much the Krylov iterative solver is able to successfully exploit a given preconditioner: higher ratios are indicative of the fact that the same factorization has been used successfully to solve a greater number of linear systems.

In Tables \ref{tab:part2tab1} and \ref{tab:part2tab2} we report the details for the largest instances of the medium-size LPs and QPs already considered in Section \ref{sec:num_res_p1} when varying the stopping tolerance ($tol$) and when the regularization parameters have been suitably increased w.r.t. the ones used in the aforementioned section. 

The results obtained confirm that the ratio $\frac{IPM It.}{Fact.}$ remains  roughly in the interval $(2.5,6)$ for all the considered problems, see also Figure \ref{fig:fact_ratio}, confirming, in general, that our proposal allows a small number of preconditioner re-computations.  Moreover, as it becomes apparent from Figure \ref{fig:fact_ratio}, when switching from $tol=10^{-5}$ to $tol=10^{-7}$ or $tol=10^{-8}$ the above mentioned ratio tends to increase for the problems \texttt{PILOT87,CVXQP1,LISWET1,LISWET10,POWELL20,SHIP12L} (the same happens for the ratio $\frac{Kryl. It.}{Fact.}$) essentially indicating that the computed preconditioners have been used  to solve successfully a larger number of linear systems during the optimization procedure. 

Indeed, this is in accordance with the observation carried out in Remark \ref{rem:1} of Section \ref{sec:preconditioning} regarding the fact that, when close enough  to convergence, less \textit{re-factorizations} are needed due to the convergence behaviour of the IPM contribution $\Theta^{-1}$ to the matrix $S(\Theta)$ (see equation \eqref{eq:Scur_Separation}).

\begin{table}[hbt]
	\centering
	\scriptsize
	\begin{tabular}{cc|cccccccc}
		\toprule
		\headcol
		LPs &   &\multicolumn{8}{c}{$tol=10^{-5}$}  \\
		\cmidrule{3-10}
		\headcol
		Problem    &    $nnz(A)$              &PPM It. & IPM It. & Kryl. It. & Fact.&   Time(s) & Obj Val  & Reg. & Status \\
		\midrule
		25FV47     &      10,705      & 26       & 26             & 1331 &  8              & 3.26            & 5501.85      & 7.00e-08 & opt \\  
		\rowcol
		80BAU3B    &     29,063      & 28       & 38             & 2447 &  12              & 11.61            & 987224.23      & 7.00e-08 & opt \\  
		D6CUBE  &     43,888           & 16       & 16             & 983 &  5              & 2.45            & 315.50      & 7.00e-08 & opt \\ 
		\rowcol
		FIT2D  &   138,018            & 23       & 23             & 1274 &  7              & 10.84            & -68464.27      & 7.00e-08 & opt \\  
		FIT2P  &    60,784             & 16       & 18             & 1002 &  5              & 5.81            & 68464.32      & 7.00e-08 & opt \\  
		\rowcol
		PILOT87&   73,804             & 33       & 39             & 2467 &  12              & 30.13            & 301.94      & 7.00e-08 & opt \\  
		QAP15 &      110,700       & 15       & 15             & 1045 &  5              & 46.36            & 1041.11      & 4.00e-06 & opt \\ 
		\bottomlinec  
		\headcol
		QPs &   &\multicolumn{8}{c}{$tol=10^{-5}$}  \\
		\cmidrule{3-10}
		\headcol
		Problem   &   $nnz(A)/nnz(H)$               &PPM It. & IPM It. & Kryl. It. & Fact.&   Time(s)&  Obj Val & Reg. & Status \\
		\midrule
		CVXQP1&    40,000/40,400     & 13       & 18             & 1161 &  5              & 23.70            & 108704555.75      & 1.00e-10 & opt  \\  
		\rowcol
		LISWET1&  30,000/10,002     & 29       & 31             & 1946 &  9              & 12.64            & 25.07      & 1.00e-09 & opt  \\ 
		LISWET10&  10,000/10,002   & 31       & 31             & 1863 &  12              & 12.18            & 25.01      & 1.00e-09 & opt  \\  
		\rowcol
		POWELL20&  20,000/10,000    & 5       & 28             & 1828 &  10              & 12.85            & 52089582811.44      & 1.00e-09 & opt  \\ 
		SHIP12L& 16,170/122,433     & 10       & 13             & 918 &  5              & 3.45            & 3018876.58      & 1.27e-09 & opt  \\
		\rowcol
		STCQP1&   13,338/49,109     & 9       & 9             & 410 &  3              & 2.22            & 155143.55      & 7.36e-07 & opt  \\
	\bottomrule	
	\end{tabular}
	\caption{Medium Size LPs and QPs results, $tol=10^{-5}$} \label{tab:part2tab1}
\end{table}

\begin{table}[hbt]
	\centering
	\scriptsize
	\begin{tabular}{cc|cccccccc}
		\toprule
		\headcol
		LPs &   &\multicolumn{8}{c}{$tol=10^{-8}$}  \\
		\cmidrule{3-10}
		\headcol
		Problem    &    $nnz(A)$              &PPM It. & IPM It. & Kryl. It. & Fact.&  Time(s) &Obj Val  & Reg. & Status \\
		\midrule
		25FV47     &       10,705         & 27       & 27             & 1332 &  9              & 3.43            & 5501.85      & 7.00e-08 & opt \\  
		\rowcol
		80BAU3B    &     29,063         & 32       & 42             & 2646 &  14              & 14.70            & 987224.19      & 7.00e-08 & opt \\    
		D6CUBE  &     43,888            & 17       & 17             & 986 &  6              & 3.00            & 315.49      & 7.00e-08 & opt \\   
		\rowcol
		FIT2D  &   138,018              & 26       & 26             & 1302 &  9              & 13.36            & -68464.29      & 7.00e-08 & opt \\ 
		FIT2P  &    60,784            & 19       & 21             & 1184 &  6              & 7.30            & 68464.29      & 7.00e-08 & opt \\  
		\rowcol
		PILOT87&   73,804              & 42       & 78             & 4870 &  21              & 62.73            & 301.80      & 7.00e-08 & opt \\ 
		QAP15 &      110,700          & 22       & 22             & 1478 &  8              & 68.24            & 1040.99      & 4.00e-06 & opt \\ 
		\bottomlinec
		\headcol
		QPs &   &\multicolumn{8}{c}{$tol=10^{-7}$}  \\
		\cmidrule{3-10}
		\headcol
		Problem   &   $nnz(A)/nnz(H)$               &PPM It. & IPM It. & Kryl. It. & Fact.& Time(s) &Obj Val & Reg. & Status \\
		\midrule
		CVXQP1&    40,000/40,400     & 14       & 19             & 1192 &  5              & 23.52            & 108704648.71      & 1.00e-10 & opt  \\  
		\rowcol
		LISWET1&  30,000/10,002      & 44       & 58             & 3584 &  11              & 22.79            & 25.12      & 1.00e-09 & opt  \\  
		LISWET10&  10,000/10,002   & 33       & 33             & 2049 &  12              & 13.47            & 25.01      & 1.00e-09 & opt  \\  
		\rowcol
		POWELL20&  20,000/10,000     & 6       & 29             & 1832 &  10              & 12.89            & 52089582812.49      & 1.00e-09 & opt  \\  
		SHIP12L& 16,170/122,433      & 28       & 33             & 1934 &  6              & 6.75            & 3018876.58      & 1.27e-09 & opt  \\  
		\rowcol
		STCQP1&   13,338/49,109    & 10       & 10             & 412 &  4              & 2.23            & 155143.55      & 7.36e-07 & opt  \\ 
		\bottomrule		
	\end{tabular}
	\caption{Medium Size LPs and QPs results, $tol=10^{-7}$ or $tol= 10^{-8}$}\label{tab:part2tab2}
\end{table}

\begin{figure}[htb!]
	\centering
	\includegraphics[width=0.8 \textwidth]{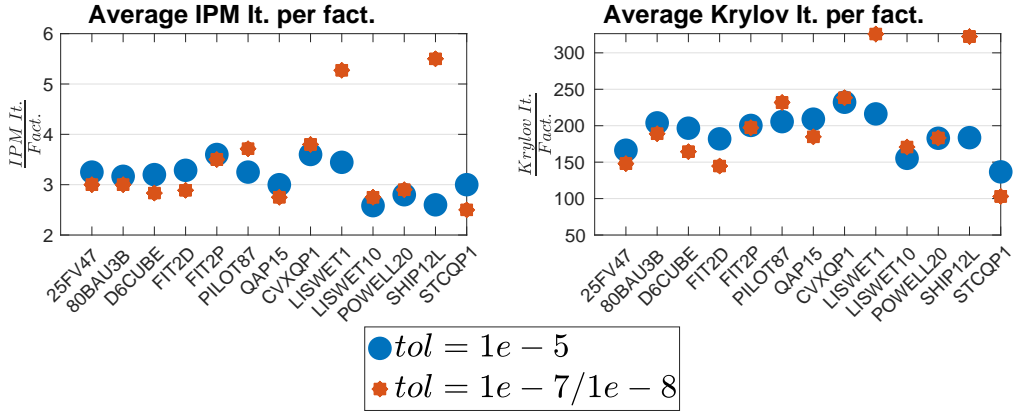}
	\caption{Average IPM and Krylov It. per factorization} \label{fig:fact_ratio}
\end{figure} 

\newpage

In Tables \ref{table:SuiteSparseIterative1} and \ref{table:SuiteSparseIterative2}  we report the details for the  instances of large size considered in Section	\ref{sec:num_res_p1} when the regularization is increased, respectively, by a factor $f=10$ and $f=500$ if compared to the regularization parameters used in Table \ref{table:1}.

In this case, the ratio $\frac{IPM \; It.}{Fact.}$ remains bounded from below by a factor strictly greater than two, clearly indicating that, also in this case, the number of necessary factorizations to optimize successfully a given problem is just a fraction of the total IPM iterations. This fact could lead to reduced computational times for instances in which the effort related to the factorization is dominant, when compared to approaches where the \texttt{ldl} factorization is recomputed at each IPM iteration in order to solve the Newton system (see, e.g., the first part of this work). 

As it becomes apparent from Figure \ref{fig:fact_suite}, when increasing the regularization parameters, the ratios  
$\frac{IPM. It.}{Fact.}$ and $\frac{Kryl. It.}{Fact.}$ tend to increase for the majority of the problems, indicating that the number of computed factorizations can be further reduced. Indeed, this is in accordance with the observation carried out in Remark \ref{rem:2} of Section \ref{sec:preconditioning} regarding the fact that the diagonal variations $(D_{A})_{ii}$ of the Schur complements $S(\Theta)$ are inversely proportional to the regularization parameters $(\rho, \delta)$ (see equation \eqref{eq:diagonal_variation}).

\begin{table}[ht!] 
	\centering
	\scriptsize
	\caption{Large Scale Problems $f=10$} \label{table:SuiteSparseIterative1}
	\begin{tabular}{ccccccccc}
		\topline
		\headcol	{Problem} &{PPM It.} &{IPM It.} & {Kryl. It.} &  {Fact.} &{Time(s)} & {Obj Val}&{Reg. Par.} & {Status} \\ \midline
	Mittelmann/fome21  & 20       & 75             & 5057 &  23              & 700.63            & 47346318912.00      & 5.43e-09 & opt \\
	\rowcol 
	LPnetlib/lp\_cre\_b  & 23       & 48             & 3760 &  16              & 81.57            & 23129639.89      & 5.00e-09 & opt \\ 
	LPnetlib/lp\_cre\_d  & 22       & 46             & 3084 &  18              & 59.21            & 24454969.78      & 5.00e-09 & opt \\ 
	\rowcol 
	LPnetlib/lp\_ken\_18  & 14       & 38             & 2241 &  14              & 215.48            & -52217025287.38      & 5.00e-09 & opt \\ 
	Qaplib/lp\_nug20  & 17       & 17             & 1056 &  8              & 310.74            & 2181.64      & 1.25e-07 & opt \\ 
	\rowcol 
	LPnetlib/lp\_osa\_30  & 19       & 29             & 1548 &  10              & 42.96            & 2142139.87      & 5.00e-09 & opt \\ 
	LPnetlib/lp\_osa\_60  & 17       & 36             & 1992 &  11              & 121.06            & 4044072.51      & 5.00e-09 & opt \\
	\rowcol  
	LPnetlib/lp\_pds\_10  & 19       & 46             & 3239 &  14              & 80.81            & 26727094976.01      & 5.43e-09 & opt \\ 
	LPnetlib/lp\_pds\_20  & 19       & 60             & 4125 &  19              & 339.66            & 23821658640.00      & 5.43e-09 & opt \\ 
	\rowcol 
	LPnetlib/lp\_stocfor3  & 32       & 35             & 1808 &  11              & 19.82            & -39976.78      & 5.00e-09 & opt \\ 
	Mittelmann/pds-100  & 20       & 85             & 5971 &  29              & 5638.99            & 10928229968.00      & 5.00e-09 & opt \\ 
	\rowcol 
	Mittelmann/pds-30  & 22       & 77             & 5087 &  23              & 709.16            & 21385445736.00      & 5.43e-09 & opt \\ 
	Mittelmann/pds-40  & 20       & 75             & 4953 &  23              & 1265.16            & 18855198824.08      & 5.43e-09 & opt \\ 
	\rowcol 
	Mittelmann/pds-50  & 19       & 78             & 5188 &  25              & 1666.61            & 16603525724.02      & 5.43e-09 & opt \\ 
	Mittelmann/pds-60  & 19       & 82             & 5909 &  26              & 2655.46            & 14265904407.03      & 5.43e-09 & opt \\ 
	\rowcol 
	Mittelmann/pds-70  & 20       & 80             & 5763 &  26              & 3511.44            & 12241162812.00      & 5.43e-09 & opt \\ 
	Mittelmann/rail2586  & 34       & 84             & 5734 &  33              & 2412.17            & 936.55      & 5.00e-09 & opt \\
	\rowcol  
	Mittelmann/rail4284  & 35       & 76             & 5353 &  27              & 2892.35            & 1054.89      & 5.00e-09 & opt \\ 
	Mittelmann/rail582  & 35       & 35             & 2461 &  11              & 56.05            & 209.75      & 5.00e-09 & opt \\ 
		\bottomlinec                    
	\end{tabular}
\end{table}

\begin{table}[ht!] 
	\centering
	\scriptsize
	\caption{Large Scale Problems $f=500$} \label{table:SuiteSparseIterative2}
	\begin{tabular}{ccccccccc}
		\topline
		\headcol	{Problem} &{PPM It.} &{IPM It.} & {Kryl. It.} &  {Fact.} &{Time(s)} & {Obj Val}&{Reg. Par.} & {Status} \\ \midline
Mittelmann/fome21  & 19       & 71             & 4757 &  22              & 656.32            & 47346318912.12      & 2.71e-07 & opt \\ 
\rowcol 
LPnetlib/lp\_cre\_b  & 24       & 50             & 3522 &  18              & 84.55            & 23129639.89      & 2.50e-07 & opt \\ 
LPnetlib/lp\_cre\_d  & 21       & 45             & 3076 &  17              & 62.66            & 24454969.77      & 2.50e-07 & opt \\ 
\rowcol 
LPnetlib/lp\_ken\_18  & 15       & 38             & 2332 &  12              & 212.58            & -52217025287.40      & 2.50e-07 & opt \\ 
Qaplib/lp\_nug20  & 17       & 17             & 1116 &  7              & 302.52            & 2181.63      & 6.25e-06 & opt \\ 
\rowcol 
LPnetlib/lp\_osa\_30  & 21       & 27             & 1581 &  10              & 45.61            & 2142139.87      & 2.50e-07 & opt \\ 
LPnetlib/lp\_osa\_60  & 21       & 34             & 1978 &  10              & 115.57            & 4044072.57      & 2.50e-07 & opt \\ 
\rowcol 
LPnetlib/lp\_pds\_10  & 20       & 46             & 3139 &  14              & 81.69            & 26727094976.05      & 2.71e-07 & opt \\ 
LPnetlib/lp\_pds\_20  & 19       & 62             & 4430 &  18              & 350.32            & 23821658639.93      & 2.71e-07 & opt \\ 
\rowcol 
LPnetlib/lp\_stocfor3  & 36       & 39             & 2107 &  10              & 28.49            & -39976.77      & 2.50e-07 & opt \\ 
Mittelmann/pds-100  & 20       & 86             & 6004 &  26              & 5614.32            & 10928229968.01      & 2.50e-07 & opt \\
\rowcol  
Mittelmann/pds-30  & 22       & 77             & 5192 &  23              & 743.52            & 21385445736.02      & 2.71e-07 & opt \\ 
Mittelmann/pds-40  & 20       & 75             & 4823 &  23              & 1216.00            & 18855198824.05      & 2.71e-07 & opt \\
\rowcol  
Mittelmann/pds-50  & 20       & 77             & 5368 &  24              & 1656.96            & 16603525724.00      & 2.71e-07 & opt \\ 
Mittelmann/pds-60  & 19       & 76             & 5207 &  25              & 2317.59            & 14265904407.01      & 2.71e-07 & opt \\
\rowcol  
Mittelmann/pds-70  & 19       & 78             & 5284 &  27              & 3192.52            & 12241162812.00      & 2.71e-07 & opt \\ 
Mittelmann/rail2586  & 34       & 82             & 5509 &  32              & 2389.66            & 936.54      & 2.50e-07 & opt \\ 
\rowcol 
Mittelmann/rail4284  & 35       & 77             & 5203 &  28              & 2799.40            & 1054.90      & 2.50e-07 & opt \\ 
Mittelmann/rail582  & 36       & 36             & 2538 &  12              & 59.42            & 209.73      & 2.50e-07 & opt \\ 
		\bottomlinec                    
	\end{tabular}
\end{table}

\begin{figure}[htb!]
	\centering
	\includegraphics[width=0.8 \textwidth]{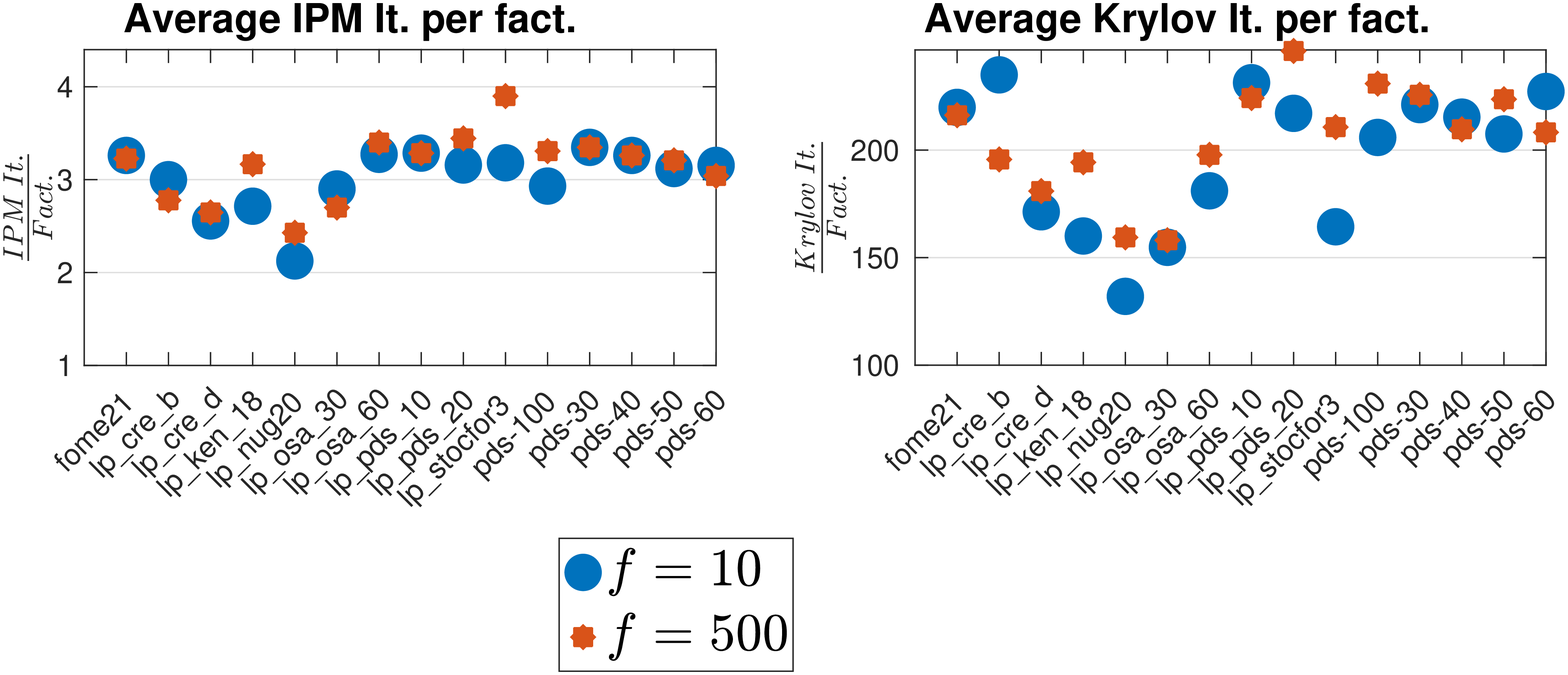}
	\caption{Average IPM and Krylov It. per factorization} \label{fig:fact_suite}
\end{figure}

Moreover, to further assess the robustness of our proposal on large scale problems, we complement the dataset used until now with additional large scale LP instances. We report in Table \ref{table:LargeSparse} the corresponding details.  

\begin{table}[ht!] 
	\centering
	\scriptsize
	\caption{Large Scale Problems, additional dataset} \label{table:LargeSparse}
	\begin{tabular}{cccccccccc}
		\topline
		\headcol	{Problem} & $nnz(A)$ &{PPM It.} &{IPM It.} & {Kryl. It.} &  {Fact.} &{Time(s)} & {Obj Val}&{Reg. Par.} & {Status} \\ \midline
Mittelmann/neos & 1,526,794 	  & 28       & 86             & 5118 &  22              & 1055.82            & 225425492.22      & 5.95e-08 & opt \\ 
\rowcol
Mittelmann/neos3 & 2,055,024 & 9       & 14             & 437 &  3              & 119.47            & 27777.78      & 5.00e-09 & opt \\ 
Mittelmann/nug08-3rd &148,416 & 8       & 8             & 613 &  3              & 322.88            & 214.00      & 7.81e-07 & opt \\ 
\rowcol
Mittelmann/stormG2\_1000  & 3,459,881 	& 20       & 94             & 5897 &  19              & 10048.66            & 15802591.43      & 5.00e-09 & opt \\ 
Meszaros/tp-6 & 11,537,419 & 27       & 34             & 2376 &  9              & 966.40 
           & -13194651.26      & 5.00e-09 & opt \\ 
\rowcol 
Mittelmann/watson\_1  &1,055,093& 31       & 90             & 
5803 &  21              & 834.08            & -1693.39      & 5.00e-09 & opt \\ 
Mittelmann/watson\_2 &1,846,391 & 31       & 104             & 6974 &  28              & 2067.68            & -7544.06      & 5.00e-09 & opt \\ 
\bottomlinec                    
	\end{tabular}
\end{table}

\subsection{PCG+\texttt{chol}}
When a further reduction to Schur complement is considered for the solution of the linear system \eqref{eq:Schur_further}, see Section \ref{sec:further_schur},  we propose to use  {PCG(200)}.
Moreover, as suggested in the discussion carried out in Section \ref{sec:further_schur}, as preconditioner of a given $L_1(\widehat{\Theta})$, we use the {Cholesky} decomposition of $L_1({\Theta})$ computed in a previous PS-IPM iteration. Analogously of what has been done in Section \ref{sec:gmres_ldl},  we factorize a given $L_1({\Theta})$ as \eqref{eq:Schur_further} using Matlab's \texttt{chol} routine  and, in our experiments, this factorization is recomputed if in the current PS-IPM step, PCG has performed more than $51\%$ of the maximum allowed iterations in the solution of at least one of the two predictor-corrector systems.  The stopping (absolute) tolerance for PCG is set as $10^{-1} \cdot tol$ (this choice does not guarantee in general the best performance, see \cite{zanetti_gondzio} for a recent analysis, but it is a robust one).

Aim of this section is to show that also in the current approach the number of necessary factorizations is equal to a fraction of the total number of IPM iterations and that a further reduction to Schur complement might improve computational times when precise criteria are met. For this reason and for the sake of brevity, we present the obtained numerical results only on a selected subset of problems
for which such reduced computational times are observed when compared to those presented in Section \ref{sec:gmres_ldl}.  

In Table \ref{table:Schur_further} we summarize the statistics of the runs of our proposal when Newton linear systems are solved with PCG+\texttt{chol}.

\begin{table}[ht!] 
	\centering
	\scriptsize
	\caption{Results obtained for PCG + \texttt{chol}} \label{table:Schur_further}
	\begin{tabular}{ccccccccc}
		\topline
		\headcol	{Problem} &{PPM It.} &{IPM It.} & {Kryl. It.} &  {Fact.} &{Time(s)} & {Obj Val}&{Reg. Par.} & {Status} \\ \midline
		\midline
		\headcol
		\multicolumn{9}{c}{Compare with Table \ref{tab:part2tab1} ($tol=10^{-5}$)}  \\
	    \midline
	    
		25FV47  & 25       & 25             & 3110 &   8              & 0.66            & 5501.85      & 7.00e-08 & opt \\ 
		80BAU3B  & 35       & 40             & 5356 &   13              & 3.39            & 987224.24      & 7.00e-08 & opt \\ 
		D6CUBE  & 16       & 16             & 1871 &   5              & 0.65            & 315.51      & 7.00e-08 & opt \\ 
		FIT2D  & 23       & 23             & 2576 &   5              & 3.23            & -68464.26      & 7.00e-08 & opt \\ 
		PILOT87  & 39       & 41             & 5498 &   13              & 7.14            & 301.95      & 7.00e-08 & opt \\ 
		\midline
		\headcol
		 \multicolumn{9}{c}{Compare with Table \ref{tab:part2tab2} ($tol=10^{-8}$)}  \\
\midline
		
		25FV47  & 28       & 28             & 3374 &   10              & 0.76            & 5501.85      & 7.00e-08 & opt \\ 
		80BAU3B  & 39       & 44             & 5599 &   15              & 4.04            & 987224.19      & 7.00e-08 & opt \\ 
		D6CUBE  & 18       & 18             & 1878 &   7              & 0.77            & 315.49      & 7.00e-08 & opt \\ 
		FIT2D  & 26       & 26             & 2773 &   6              & 3.58            & -68464.29      & 7.00e-08 & opt \\ 
		PILOT87  & 49       & 95             & 11780 &   20              & 14.96            & 301.79      & 7.00e-08 & opt \\ 
		\midline
		\headcol
		\multicolumn{9}{c}{Compare with Table \ref{table:SuiteSparseIterative1} ($tol=10^{-6}$, $f=10$)}  \\
	
\midline
		LPnetlib/lp\_osa\_30  & 27       & 27             & 3242 &  7              & 13.95            & 2142139.88      & 5.00e-09 & opt  \\  
		LPnetlib/lp\_osa\_60  & 26       & 34             & 4042 &  8              & 48.25            & 4044072.51      & 5.00e-09 & opt  \\  
		LPnetlib/lp\_pds\_10  & 25       & 47             & 5601 &  16              & 61.05            & 26727095000.48      & 5.43e-09 & opt  \\  
		LPnetlib/lp\_stocfor3  & 34       & 34             & 4447 &  11              & 7.48            & -39976.78      & 5.00e-09 & opt  \\  
		Mittelmann/rail2586  & 40       & 90             & 10308 &  20              & 576.67            & 936.58      & 5.00e-09 & opt  \\  
		Mittelmann/rail4284  & 41       & 86             & 9717 &  20              & 872.56            & 1054.62      & 5.00e-09 & opt  \\  
		Mittelmann/rail582  & 36       & 36             & 5003 &  9              & 10.03            & 209.76      & 5.00e-09 & opt  \\  
		
		\midline
		\headcol
		\multicolumn{9}{c}{Compare with Table \ref{table:SuiteSparseIterative2} ($tol=10^{-6}$, $f=500$)}  \\
	\midline

		LPnetlib/lp\_osa\_30  & 28       & 28             & 3014 &  8              & 15.79            & 2142139.87      & 2.50e-07 & opt  \\  
		LPnetlib/lp\_osa\_60  & 27       & 42             & 4894 &  9              & 62.70            & 4044072.53      & 2.50e-07 & opt  \\  
		LPnetlib/lp\_pds\_10  & 25       & 47             & 6544 &  15              & 62.05            & 26727094975.92      & 2.71e-07 & opt  \\  
		LPnetlib/lp\_stocfor3  & 41       & 41             & 5300 &  12              & 8.29            & -39976.77      & 2.50e-07 & opt  \\  
		Mittelmann/rail2586  & 40       & 94             & 11200 &  21              & 619.89            & 936.60      & 2.50e-07 & opt  \\  
		Mittelmann/rail4284  & 41       & 79             & 9752 &  19              & 898.94            & 1054.80      & 2.50e-07 & opt  \\  
		Mittelmann/rail582  & 37       & 37             & 4836 &  10              & 10.09            & 209.75      & 2.50e-07 & opt  \\  
		\midline
		\headcol
		\multicolumn{9}{c}{Compare with Table \ref{table:LargeSparse} ($tol=10^{-6}$)}  \\
		\midline
		Meszaros/tp-6  & 31	  & 34             & 4325 &   9              & 758.84            & -13194651.17      & 5.00e-09 & opt \\
		Mittelmann/watson\_1  & 35	& 84             & 10786 &   24              & 331.30            & -1693.28	 & 5.00e-09 & opt \\
		Mittelmann/watson\_2  & 36	& 80             & 9055 &   21              & 481.80            & -7543.88	& 5.00e-09 & opt \\ 
		
		\bottomlinec                    
	\end{tabular}
\end{table}

As the results presented in Table \ref{table:Schur_further} confirm, in LP problems for which the  pattern of the matrix $A^TA+\delta I$ is particularly sparse and/or such matrix is of small dimension, the approach considered in this section leads to improved computational times while performing a limited number of Cholesky factorizations. To further underpin this point we report in Figure \ref{fig:chol_spar} the sparsity of some Cholesky factors for which improved computational times are observed when compared to those presented in Section \ref{sec:gmres_ldl}.

\begin{figure}[htb!]
	\centering
	\includegraphics[width=1 \textwidth]{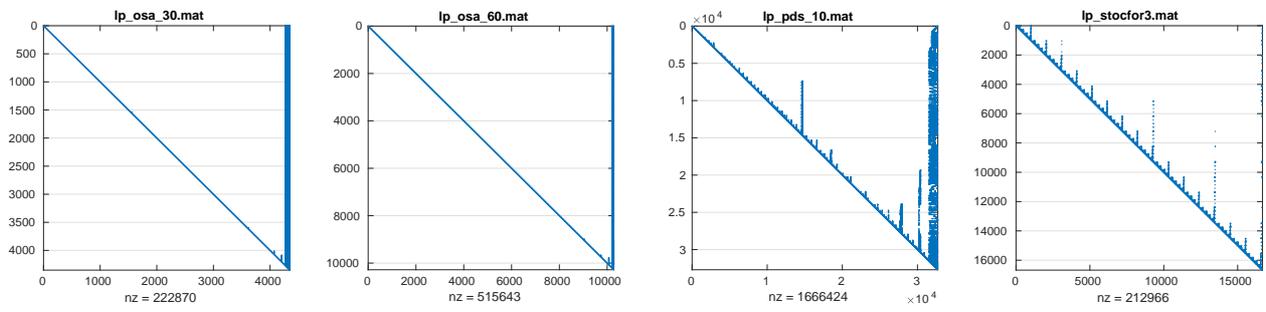}
	\caption{Sparsity  patterns of the Cholesky factors of $A^TA+\delta I$} \label{fig:chol_spar}
\end{figure}

\newpage

\section{Conclusions}

In this work we have clarified certain nuances of the convergence theory of primal-dual regularized Interior Point Methods (IPMs) using the inexact Proximal Point Method (PPM) framework: if on one hand this closes an existing literature gap, on the other, it sheds further light on their optimal implementation especially in the (nearly) rank deficient case of the linear constraints and/or in the large scale setting. 

Indeed, our convergence analysis does not require any linear independence assumption on the linear constraints nor the positive definiteness of the quadratic term. Moreover, when a direct solver can be used for the solution of the Newton system, we showed  experimentally in Section \ref{sec:num_res_p1} that a fixed but small  regularization parameter is preferred to strategies in which the regularization is driven to zero. The second part of this work has been devoted to the study of the interactions between the regularization parameters and the computational complexity of the linear algebra solvers used in IPM. In Section \ref{sec:preconditioning} we proposed a new preconditioning technique able to exploit regularization as a tool to reduce the number of preconditioner re-computations  when an iterative solver is needed for the solution of the IPMs Newton systems. Indeed, we were able to devise a class of general purposes preconditioners which require an \textit{update frequency} inversely proportional to the magnitude of the regularization parameters.  Finally, building the momentum from the correct interpretation of the primal-dual regularization parameters in connection with the overall rate of convergence of the PPM and the proposed preconditioning strategy, we were able to show the robustness and efficiency of our proposal on a class of medium and large scale LPs and QPs.

\printbibliography	
\end{document}